\documentclass[10pt]{article}
\usepackage{bbm}

\usepackage{amsmath,amssymb,amsthm,color,epsfig,mathrsfs}
\usepackage{amsbsy,soul}
\usepackage{cancel}

\newlength\enumsep
\newcounter{enumlist}

\newtheorem{theorem}{Theorem}

\newtheorem{condition}[theorem]{Condition}

\newtheorem{corollary}[theorem]{Corollary}

\newtheorem{definition}[theorem]{Definition}

\newtheorem{lemma}[theorem]{Lemma}

\newtheorem{proposition}[theorem]{Proposition}

\newcounter{spslist}

\newcounter{geqncount}
    {\refstepcounter{equation}%
     \setcounter{geqncount}{\value{equation}}%
     \setcounter{equation}{0}%
  }%
    {\setcounter{equation}{\value{geqncount}}}

\newcommand{\pp}{\text{p}}

\newcommand{\RR}{\mathbb{R}}
\newcommand{\CC}{\mathbb{C}}

\newcommand{\half}{\frac{1}{2}}

\newcommand{\Hhalf}{H^{1/2}} 
\newcommand{\Hmhalf}{H^{-1/2}} 
\newcommand{\Hmhalfz}{H^{-1/2}_0} 
\newcommand{\Hmhalfo}{H^{-1/2,o}} 
\newcommand{\Hmhalfe}{H^{-1/2,e}} 

\newcommand{\Gz}{\Gamma_{\!0}}

\newcommand{\KsGz}{\mathcal K_{\Gamma_{\!0}}^*}

\newcommand{\KsG}{\mathcal K_{\Gamma}^*}
\newcommand{\KG}{\mathcal K_{\Gamma}}

\newcommand{\KsGzo}{\mathcal K_{\Gamma_{\!0},o}^*}
\newcommand{\KsGo}{\mathcal K_{\Gamma,o}^*}
\newcommand{\KGo}{\mathcal K_{\Gamma,o}}

\newcommand{\KsGze}{\mathcal K_{\Gamma_{\!0},e}^*}
\newcommand{\KsGe}{\mathcal K_{\Gamma,e}^*}
\newcommand{\KGe}{\mathcal K_{\Gamma,e}}

\newcommand{\ksGz}{K_{\Gamma_{\!0}}^*}
\newcommand{\ksG}{K_{\Gamma}^*}

\newcommand{\KsS}{\mathcal K_{\Sigma}^*}
\newcommand{\KS}{\mathcal K_{\Sigma}}
\newcommand{\KsSo}{\mathcal K_{\Sigma,o}^*}
\newcommand{\KSo}{\mathcal K_{\Sigma,o}}
\newcommand{\KsSe}{\mathcal K_{\Sigma,e}^*}
\newcommand{\KSe}{\mathcal K_{\Sigma,e}}
\newcommand{\ksS}{K_{\Sigma}^*}

\newcommand{\KsO}{\mathcal K_{\partial\Omega}^*}

\newcommand{\KsOo}{\mathcal K_{\partial\Omega,o}^*}
\newcommand{\KOo}{\mathcal K_{\partial\Omega,o}}
\newcommand{\KsOe}{\mathcal K_{\partial\Omega,e}^*}
\newcommand{\KOe}{\mathcal K_{\partial\Omega,e}}

\newcommand{\SG}{\mathcal{S}_\Gamma}
\newcommand{\SGz}{\mathcal{S}_{\Gamma_{\!0}}}
\newcommand{\CS}{C_{\hspace{-1pt}\mathcal{S}}}
\newcommand{\Cphi}{C_{\hspace{-1pt}\phi}}

\newcommand{\disk}{\Delta}
\newcommand{\type}{$T$ }

\newcommand{\ess}{\text{\hspace{-1pt}ess}}
\newcommand{\acont}{\text{\hspace{-1pt}ac}}
\newcommand{\scont}{\text{\hspace{-1pt}sc}}
\renewcommand{\pp}{\text{\hspace{-1pt}pp}}

\newcommand{\meas}{\mathrm{len}}

\newcommand{\dist}{\mathrm{dist}}
\newcommand{\supp}{\mathrm{supp}}

\usepackage{changes}
\begin{document}

\bibliographystyle{plain} 

\begin{center}
{\bf \Large Embedded eigenvalues for the Neumann-Poincar\'e operator}
\end{center}

\vspace{0.2ex}

\begin{center}
{\scshape \large Wei Li  \,\,\,and\,\,\, Stephen P\!. Shipman} \\
\vspace{1ex}
{\itshape Department of Mathematics}\\
{\itshape Louisiana State University, Baton Rouge, LA, USA}
\end{center}

\vspace{3ex}
\centerline{\parbox{0.9\textwidth}{
{\bf Abstract.}\
The Neumann-Poincar\'e operator is a boundary-integral operator associated with harmonic layer potentials.  This article proves the existence of eigenvalues within the essential spectrum for the Neumann-Poincar\'e operator for certain Lipschitz curves in the plane with reflectional symmetry, when considered in the functional space in which it is self-adjoint.
The proof combines the compactness of the Neumann-Poincar\'e operator for curves of class $C^{2,\alpha}$ with the essential spectrum generated by a corner.  Eigenvalues corresponding to even (odd) eigenfunctions are proved to lie within the essential spectrum of the odd (even) component of the operator when a $C^{2,\alpha}$ curve is perturbed by inserting a small corner.
}}

\vspace{3ex}
\noindent
\begin{mbox}
\end{mbox}
\vspace{3ex}

\hrule
\vspace{1.1ex}

\section{Introduction and basics of the Neumann-Poincar\'e operator}\label{sec:introduction}

The Neumann-Poincar\'e operator $\KG$ and its formal adjoint $\KsG$ are boundary-integral operators associated with the double-layer harmonic potential and the normal derivative of the single-layer harmonic potential for the boundary $\Gamma$ of a bounded domain in $\RR^n$.  When $\Gamma$ is of class $C^2$, these operators are compact, and thus their spectra consist only of eigenvalues converging to zero (and zero itself).  For domains with Lipschitz boundary, they have essential spectrum, which depends critically on the function spaces in which they act.  This work proves the existence of eigenvalues within the essential spectrum of $\KsG$ for certain Lipschitz curves $\Gamma$ in $\RR^2$ in the Sobolev distribution space $\Hmhalf(\Gamma)$, in which $\KsG$ is self-adjoint (Theorem~\ref{thm:main}).  The theorem implies eigenvalues within the essential spectrum for $\KG$ in $\Hhalf(\Gamma)$, which has exactly the same spectrum as $\KsG$ in $\Hmhalf(\Gamma)$.

In $\RR^2$, if $\Gamma$ is the boundary of a simply connected bounded domain, the Neumann-Poincar\'e operator applied to a function $\phi:\Gamma\to\CC$  is 
\begin{equation}
  \KG[\phi](x) \;=\; - \frac{1}{2\pi} \int_\Gamma \phi(y) \frac{x-y}{|x-y|^2}\cdot n_y\, ds(y),
\end{equation}
in which $x$ and $y$ are on $\Gamma$, $n_y$ is the outward-directed normal vector to $\Gamma$ at $y\in\Gamma$, and $ds(y)$ is the arclength measure on $\Gamma$.  The adjoint of $\KG$ in $L^2(\Gamma)$, which we called the formal adjoint $\KsG$ above, is
\begin{equation}
  \KsG[\phi](x) \;=\; \frac{1}{2\pi} \int_\Gamma \phi(y) \frac{x-y}{|x-y|^2}\cdot n_x\, ds(y).
\end{equation}
These operators are defined as legitimate integrals when $\Gamma$ and $\phi$ are smooth enough, and they are extended to different normed distributional spaces by continuity.

The eigenvalues of $\KsG$ in $L^2(\Gamma)$ are real.  This is because $\KsG$ is symmetric with respect to the inner product associated with a weaker norm defined through the boundary-integral operator $\SG$ for the single-layer potential,
\begin{equation}\label{eq:single}
  \SG[\phi] (x) \,=\, - \frac{1}{2\pi}\int_{\Gamma} \log (\beta |x-y|)\, \phi(y)\, ds(y).
\end{equation}
For appropriately chosen $\beta>0$,
this operator on $L^2(\Gamma)$ is strictly positive~\cite[Lemma~2.1]{KhavinsonPutinarShapiro2007} and not surjective since it is bounded and invertible from $\Hmhalf(\Gamma)$ to $\Hhalf(\Gamma)$~\cite{CostabelStephan1985,PerfektPutinar2014}.  The Plemelj symmetrization principle 
\begin{equation}
  \KG\SG = \SG\KsG
\end{equation}
in $L^2(\Gamma)$ implies the symmetry of $\KsG$ with respect to the inner product $\langle f, g \rangle_{\SG} := (\SG f,g)_{L^2(\Gamma)}$,
\begin{equation}\label{eq:sinner}
  \langle \KsG f,g \rangle_{\SG}
  = \langle f, \KsG g \rangle_{\SG}\,.
\end{equation}
Perfekt and Putinar~\cite{PerfektPutinar2014} show that this theory persists even for Lipschitz curves~$\Gamma$.
By completing the vector space $L^2(\Gamma)$ with respect to the $\mathcal{S}$ norm
\begin{equation} \label{eq:snorm}
  \|f\|^2_{\SG}=\langle \SG f, f \rangle_{L^2(\Gamma)},
\end{equation}
$\KsG$ is extended by continuity to a self-adjoint operator.  This completion space coincides with the Sobolev space $\Hmhalf(\Gamma)$ of distributions~\cite[Lemma~3.2]{PerfektPutinar2014}, which is sometimes referred to as the ``energy space" for~$\KsG$.
In this article, $\Hmhalf(\Gamma)$ will always refer to the Hilbert space with the $\mathcal{S}$ inner product $\langle f, g \rangle_{\SG}$.

The operator norm of $\KsG$, as a self-adjoint operator in $\Hmhalf(\Gamma)$, is equal to $1/2$, and the spectrum is contained in the half-open interval $(-1/2,1/2]$, with $1/2$ being an eigenvalue.  The eigenspace is spanned by the density for a single-layer potential that is constant in the interior domain of $\Gamma$~\cite{Kellogg1929}.

The analogous space in which $\KG$ is self-adjoint is~$\Hhalf(\Gamma)\!\subset\!L^2(\Gamma)$
with respect to the norm $(\SG^{-1} f,g)_{L^2(\Gamma)}$.
Therefore, any eigenfunction of $\KG$ corresponding to a non-real eigenvalue $\lambda$ cannot lie in $\Hhalf(\Gamma)$.  When $\Gamma$ is a curvilinear polygon, $\KG$ does admit non-real eigenvalues with eigenfunctions in $L^2(\Gamma)$.  Mitrea~\cite{Mitrea2002} proved that these eigenvalues fill the interior domains of bowtie-shaped curves in the complex plane that are symmetric about the real line, one for each corner.  The curves themselves consist of essential spectrum.  The operator~$\KsG$, on the other hand, being self-adjoint in~$\Hmhalf(\Gamma)$ with respect to the inner product $\langle f, g \rangle_{\SG}$, cannot have non-real eigenvalues with eigenfunctions in $L^2(\Gamma)\!\subset\!\Hmhalf(\Gamma)$.  This means that, for a non-real eigenvalue $\lambda$ of $\KG$, the operator $\KsG-\bar\lambda I$ acting on $L^2(\Gamma)$ is injective and has range that is not dense in $L^2(\Gamma)$; such $\bar\lambda$ is in the residual spectrum of $\KsG$ as an operator on $L^2(\Gamma)$.
Helsing and Perfekt~\cite{HelsingPerfekt2017} proved that, for a domain in $\RR^3$ with a single conical point and continuous rotational symmetry, this spectrum consists of an infinite union of conjugate-symmetric domains in the complex plane corresponding to the Fourier components.

In $\Hmhalf(\Gamma)$, where $\KsG$ is self-adjoint, the essential spectrum of $\KsG$ for a curvilinear polygon consists of an interval in the real line that is symmetric about $0$~\cite{Carleman1916,PerfektPutinar2014,PerfektPutinar2017}.  Each corner of~$\Gamma$ contributes an interval $[-b,b]$ to the essential spectrum, and $b$ varies monotonically between $0$ and $1/2$ as the corner becomes sharper, as described in section~\ref{sec:essential}.  When the corner is outward-pointing and $\Gamma$ has reflectional symmetry about a line~$L$ with the tip of the corner on $L$, the interval $[-b,0]$ is the essential spectrum for the odd component of $\KsG$ and $[0,b]$ is the essential spectrum for the even component~\cite{KangLimYu2017}.  When the corner is inward-pointing, this correspondence is switched.  This separation of even and odd essential spectrum is critical in our proof of eigenvalues within the essential spectrum.  

In his 1916 dissertation~\cite{Carleman1916}, Torsten Carleman considered eigenfunctions of the operator $\KsG$ that are orthonormal with respect to the $\mathcal{S}$ inner product (p.~157--178 and equation~(194)),  as well as generalized eigenfunctions for a curve with corners.  At the end of the work (p.~193), he writes a spectral representation for $\KsG g$ in terms of a sum over eigenfunctions plus an integral over generalized eigenfunctions, for functions $g$ that have finite $\mathcal{S}$ norm.  The validity of this analysis for $\KsG$ in the space $\Hmhalf(\Gamma)$ would establish the absolute continuity of the essential spectrum associated with the generalized eigenfunctions, which causes the eigenvalues of our Theorem~\ref{thm:main} to be embedded in the continuous spectrum.  It is strongly believed, if not generally accepted, that the essential spectrum and the absolutely continuous spectrum coincide.

There is numerical evidence of embedded eigenvalues for the Neumann-Poincar\'e operator.  Helsing, Kang, and Lim~\cite{HelsingKangLim2016} numerically implement a rate-of-resonance criterion and illustrate two eigenvalues within the continuum for an ellipse with an attached corner.  We will revisit this example in discussion point~(5) of section~\ref{sec:discussion}.  For a rotationally symmetric domain in $\RR^3$ with a conical point mentioned above~\cite[\S7.3.3,~Fig.~8]{HelsingPerfekt2017}, eigenvalues for certain Fourier components of the Neumann-Poincar\'e operator are computed, and these lie within the essential spectrum of other Fourier components.

Our strategy for proving eigenvalues in the essential spectrum goes as follows.  Start with a curve $\Gz$ that is of class~$C^{2,\alpha}$ and that is reflectionally symmetric about a line $L$.  Let $\lambda$ be an eigenvalue of $\KsGz$ that is, say, positive with eigenfunction that is, say, odd with respect to $L$.  Then construct a symmetric perturbation $\Gamma$ of $\Gz$ such that (1) $\KsG$ has a positive eigenvalue near $\lambda$ with odd eigenfunction and (2) the even component of $\KsG$ has essential spectrum that overlaps this eigenvalue.  To accomplish the second requirement, $\Gamma$ is constructed by replacing a small segment of $\Gz$ with a corner that connects smoothly to the rest of $\Gz$, with the tip of the corner lying on $L$ and whose angle is such that $\lambda\in(0,b)$.  To accomplish the first requirement, the replaced segment needs to be sufficiently small.  The analysis of requirement (1) is remarkably subtle, and our proof relies on the deep fact that all eigenfunctions of $\KsGz$ as an operator in $\Hmhalf(\Gz)$ are actually in $L^2(\Gz)$.

Perturbative spectral analysis of $\KsG$ in $\Hmhalf(\Gamma)$ relies on the self-adjointness of the operators $\KsG$ in the $\mathcal{S}$ inner product.  But the positive-definiteness of this inner product requires an appropriate choice of the constant $\beta$ in (\ref{eq:single}), and this depends on the domain $\Gamma$.  As $\Gamma$ varies over a family of Lipschitz perturbations of a smooth curve, it must be guaranteed that $\mathcal{S}$ remain positive for all perturbations.  Instead of controlling the number $\beta$, this inconvenience can be dealt with by restricting to the $\KsG$-invariant subspace $\Hmhalfz(\Gamma)$, on which $\langle \cdot,\,\cdot \rangle_{\mathcal S}$ remains positive.  The space $\Hmhalfz(\Gamma)$ consists of all distributions $\psi\in\Hmhalf(\Gamma)$ such that $\langle \psi,1 \rangle=0$ in the $\Hmhalf$-$\Hhalf$ pairing.  The $\mathcal{S}$-orthogonal complement of $\Hmhalfz(\Gamma)$ in $\Hmhalf(\Gamma)$ is spanned by the eigenfunction of $\KsG$ corresponding to eigenvalue $1/2$.
Some interesting aspects of the definiteness of the single-layer potential in two dimensions are investigated in~\cite{Zoalroshd2016}.

\section{Approximate eigenfunction on a perturbed curve}\label{sec:perturbation}

This section accomplishes the first step, which is to construct an approximate eigenfunction $\tilde\phi$ of $\KsG$ for a Lipschitz perturbation $\Gamma$ of a $C^{2,\alpha}$ curve~$\Gz$.
The strategy is as follows.  Start with a curve $\Gz$ of class $C^{2,\alpha}$ and an eigenfunction $\phi$ of $\KsGz$ as an operator in $\Hmhalf(\Gz)$, that is, $\KsGz\phi = \lambda\phi$.  Then construct a Lipschitz perturbation $\Gamma$ of $\Gz$ by replacing a segment of $\Gz$ by a curve with a corner so that the restriction $\tilde\phi$ of $\phi$ to the rest of the curve---which is common to both $\Gz$ and $\Gamma$---is nearly an eigenfunction of~$\KsG$ in the sense that
$
\|(\mathcal K_{\Gamma}^* - \lambda )\tilde\phi \|_{\SG}
 \leq \epsilon\,\|\tilde\phi\|_{\SG}
$.
This is the essence of the proof of Lemma~\ref{lemma:resolvent}, which concludes that the resolvent $(\KsG-\lambda)^{-1}$ can be made as large as desired by taking a fine enough perturbation $\Gamma$.

Our proof of Lemma~\ref{lemma:resolvent} relies on the fact that any eigenfunction of $\KsGz:\Hmhalf(\Gz)\to\Hmhalf(\Gz)$ actually lies in $L^2(\Gz)$.  This was observed by Khavinson, Putinar, and Shapiro~\cite{KhavinsonPutinarShapiro2007,Putinar2017}, in which a theory of M.\,Krein~\cite{Krein1947,Krein1998} on operators in the presence of two norms was brought to bear on the Neumann-Poincar\'e operator.  
Lemma~\ref{lemma:L2} is essentially Theorem~3 of~\cite{Krein1998}.
We include a proof here.

\begin{lemma}\label{lemma:L2}
Let $\Gz$ be a simple closed curve of class $C^2$ in $\RR^2$.
If $\phi\in\Hmhalf(\Gz)$ satisfies $\KsGz\phi=\lambda\phi$ for a nonzero real number $\lambda$, then $\phi\in L^2(\Gz)$.
\end{lemma}

\begin{proof}
Let $\beta$ in the kernel of $\SGz$ (\ref{eq:single}) be chosen such that $\langle \cdot,\,\cdot \rangle_{\SGz}$ is positive definite on $\Hmhalf(\Gz)$.
Let $\lambda$ be a nonzero real number.  Let $N$ denote the nullspace of $\KsGz-\lambda I$ in $L^2(\Gz)$, and let $V$ denote its complement with respect to the inner product induced by the single-layer operator $\SGz$,
\begin{align}
  N &:= \; \big\{ f\in L^2(\Gz)  :  (\KsGz \!-\! \lambda I)f = 0 \big\},  \\
  V &:= \; \big\{ f\in L^2(\Gz) : \langle f,\,g \rangle_{\SGz} = 0\; \forall\,g\in N \big\}.
\end{align}
The space $V$ is closed in $L^2(\Gz)$, and $L^2(\Gz)=N+V$ as an algebraic direct sum.
The operator $\KsGz\!-\!\lambda I$ is invariant on $V$ because of the symmetry of $\KsGz$ with respect to $\langle \cdot,\cdot \rangle_{\SGz}$.  Its restriction to $V$ is injective and $\KsGz$ restricted to $V$ is compact in the $L^2(\Gz)$ norm because $\KsGz$ is compact in $L^2(\Gz)$~\cite{CostabelStephan1985,PerfektPutinar2014}.  This implies that 
$\KsGz\!-\!\lambda I$ is surjective on $V$, using the fact that the Fredholm index of $\KsGz\!-\!\lambda I$ on $V$ is zero.
Therefore $(\KsGz\!-\!\lambda I)^{-1} : V\to V$ exists as a bounded operator in the $L^2(\Gz)$ norm with
$(\KsGz\!-\!\lambda I)^{-1} (\KsGz\!-\!\lambda I)$ being the identity operator on~$V$.

The symmetry of $\KsGz$ with respect to $\langle \cdot,\cdot \rangle_{\SGz}$ implies that $(\KsGz\!-\!\lambda I)^{-1}$ is also symmetric with respect to this inner product.  The key step of the proof is now an application of Theorem~1 in~\cite{Krein1998}.  Since the $\mathcal{S}$~norm is weaker than the $L^2$ norm, this symmetry implies that $(\KsGz\!-\!\lambda I)$ and $(\KsGz\!-\!\lambda I)^{-1}$ are bounded when considered as operators in $V$, viewed as an incomplete normed linear space with respect to $\|\cdot\|_{\SGz}$.
Since $\|\cdot\|_{\SGz}$ is equivalent to the $\Hmhalf(\Gz)$ norm, $\KsGz\!-\!\lambda I$ and $(\KsGz\!-\!\lambda I)^{-1}$ extend uniquely to the completion $\tilde V$ of $V$ in $\Hmhalf(\Gz)$, and the composition
$(\KsGz\!-\!\lambda I)^{-1} (\KsGz\!-\!\lambda I)|_V$ lifts to the identity operator on $\tilde V$~\cite[Theorem~2]{Krein1998}.

Since $N$ is finite dimensional and $L^2(\Gz)=N+V$, one has $\Hmhalf(\Gz)=N+\tilde V$.
And since $\KsGz\!-\!\lambda I$ is invertible on $\tilde V$ and $(\KsGz\!-\!\lambda I)[N]=\{0\}$, it follows that the nullspace of $\KsGz-\lambda I$ in $\Hmhalf(\Gz)$ is equal to~$N$,
\begin{equation}
  \big\{ f\in \Hmhalf(\Gz) : (\KsGz\!-\!\lambda I)f=0 \big\} \,=\, N.
\end{equation}
This implies that every eigenfunction $\KsGz$ that is in $\Hmhalf(\Gz)$ also lies in $L^2(\Gz)$.
\end{proof}

If the curve $\Sigma$ (which could be either $\Gz$ or $\Gamma$) admits reflection symmetry about a line $L$, one has a decomposition
\begin{equation}
  \Hmhalf(\Sigma) = \Hmhalfe(\Sigma) \oplus \Hmhalfo(\Sigma)
\end{equation}
into spaces of even and odd distributions with respect to $L$.  This is an orthogonal direct sum with respect to the $\mathcal{S}$ inner product.
Since the operator $\KsS$ commutes with reflection symmetry, this decomposition of $\Hmhalf(\Sigma)$ induces a decomposition of $\KsS$ onto the even and odd distribution spaces, on which it is invariant,
\begin{equation}\label{eq:Kdecomposition}
  \KsS = \KsSe\oplus\KsSo.
\end{equation}

The Lipschitz perturbations of $\Gz$ and near-eigenfunctions constructed in this section have to be controlled in a careful way.  We therefore make a precise definition of the type of perturbation we will use.  It is by no means the most general.  The specific geometry of the corner is not important but serves to simplify the proofs; indeed, the invariance of the essential spectrum under smooth perturbations of a Lipschitz curve that preserve the angles of the corners is proved in~\cite[Lemma~4.3]{PerfektPutinar2014}.  The perturbed curves $\Gamma$ constructed in Definition~\ref{def:Tperturbation} have corners that are locally identical to a corner of a prototypical simple closed Lipschitz curve featuring a desired half exterior angle $\theta$ with $0<\theta<\pi$.  This curve  is the boundary $\partial\Omega$ of a region $\Omega$ defined by two intersecting circles of the same radius, as illustrated in Fig.~\ref{fig:disks}.  Explicit spectral analysis of these domains has been carried out by Kang, Lim, and Yu~\cite{KangLimYu2017} and will be used in the analysis in section~\ref{sec:essential}.

\begin{figure}  
  \centerline{\scalebox{0.33}{\includegraphics{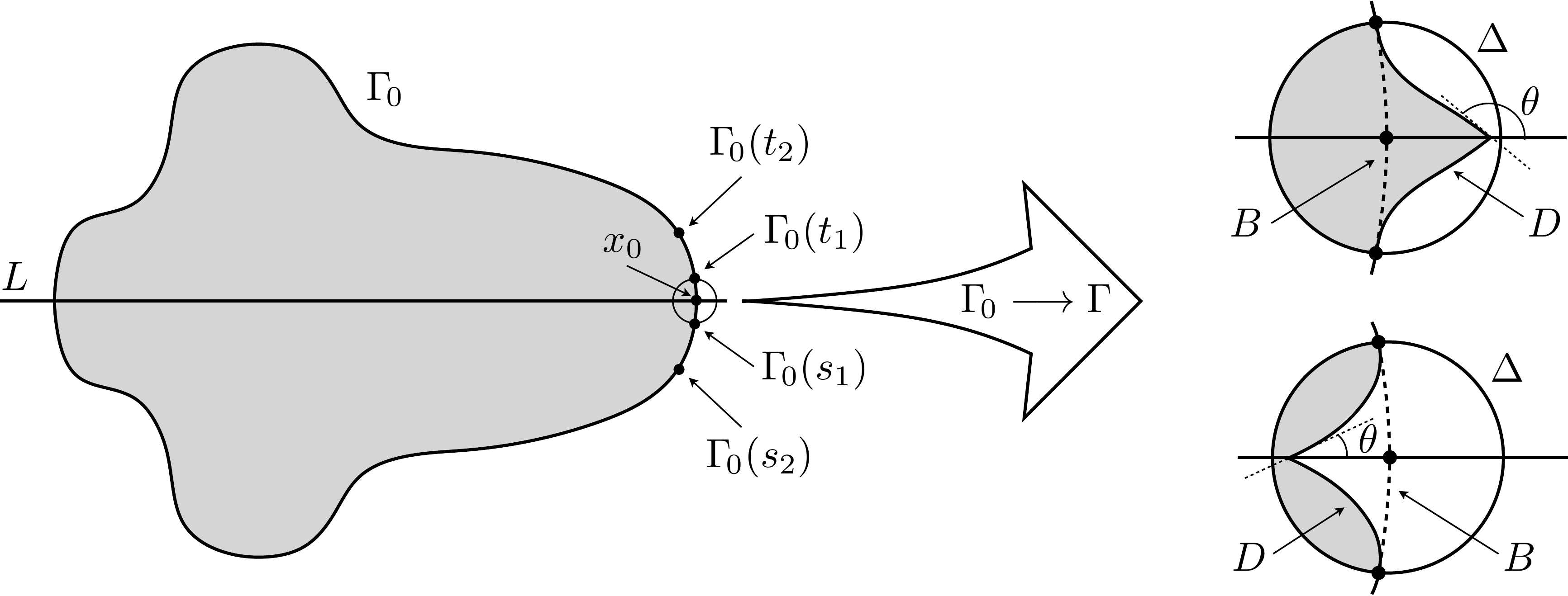}}}
\caption{\small A type \type perturbation of a curve $\Gz$ of class $C^{2,\alpha}$, as described in Definition~\ref{def:Tperturbation}, with reflectional symmetry about the line $L$.  The segment $B$ of $\Gz$ that is contained in the disk $\Delta$ is replaced by a curve with a corner to obtain~$\Gamma$.  In the upper case where the half exterior angle satisfies $\pi/2<\theta<\pi$, the corner is pointing outward; and in the lower case where $0<\theta<\pi/2$, the corner is pointing inward.  The curve $\Gz$ is parameterized by the interval $[0,1]$ with $\Gz(0)=\Gz(1)=x_0$ and $0<t_1<t_2<s_2<s_1<1$.}
\label{fig:Tperturbation}
\end{figure}

\begin{figure}
  \centerline{\scalebox{0.3}{\includegraphics{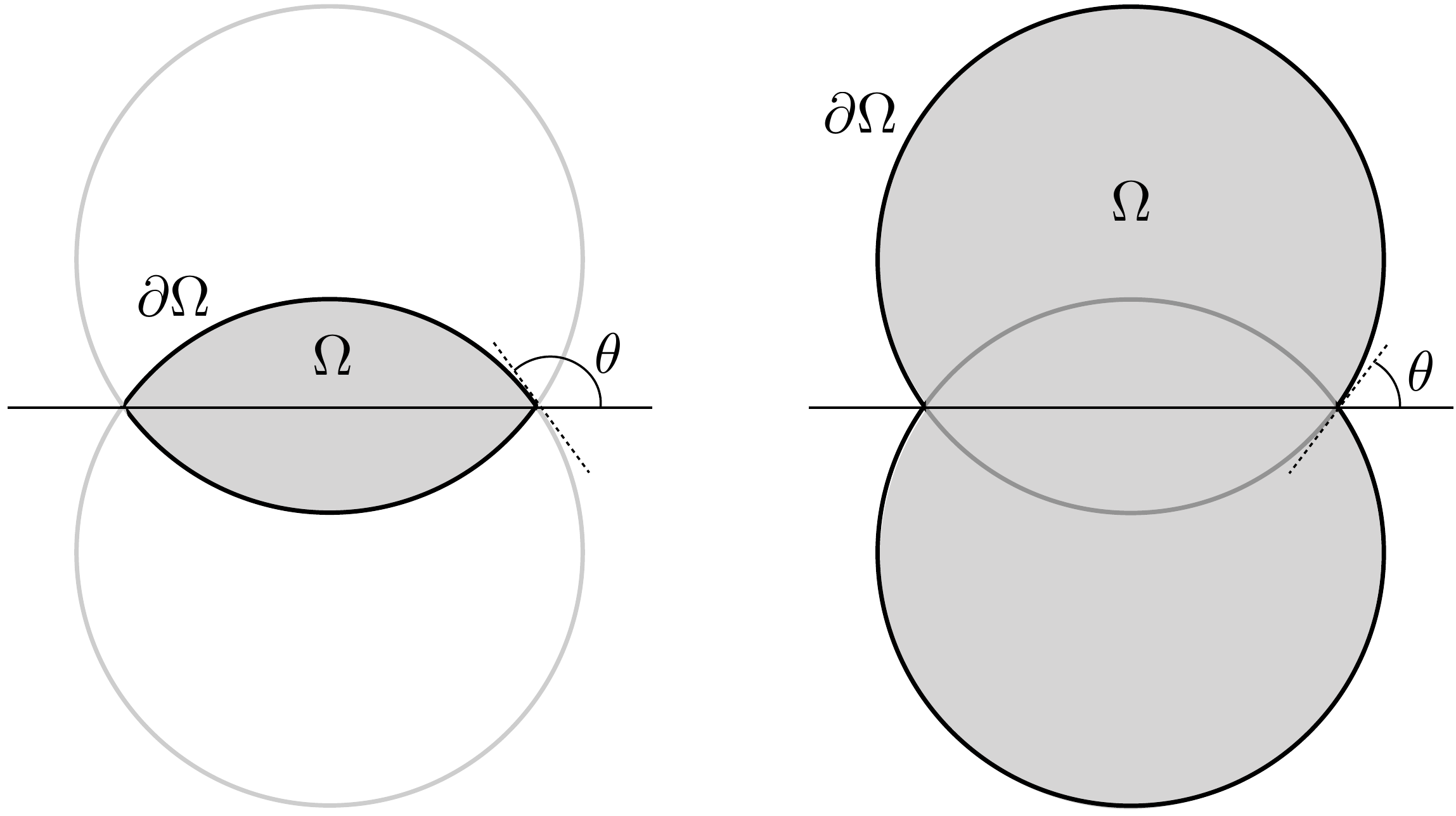}}}
\caption{\small The boundary $\partial\Omega$ of a bounded domain $\Omega$ defined by two intersecting circles of the same radius is the prototype of a curvilinear polygon.  On the left, the outward-pointing corner has half exterior angle $\theta: \pi/2<\theta<\pi$; and on the right, the inward-pointing corner has half exterior angle $\theta: 0<\theta<\pi/2$.}
\label{fig:disks}
\end{figure}

\begin{definition}\label{def:Tperturbation}
Let $\Gz$ be a simple closed curve of class $C^{2,\alpha}$ ($\alpha>0$) in $\RR^2$.
 A {\em type \type perturbation} of $\Gz$ is a curve $\Gamma$ that has one corner with half exterior angle given arbitrarily by $\theta : 0<\theta<\pi$ and is otherwise of class $C^{2,\alpha}$, and that is equipped with the following structure.

\smallskip
(a) Let $x_0\in\Gz$ be a reference point, and let $\Gz$ be parameterized by the unit interval $[0,1]$ (using the notation $\Gz(t)$ for $t\in[0,1]$) with $\Gz(0)=\Gz(1)=x_0$.

\smallskip
(b) Let $\disk=\left\{ x : | x - x_0 | \leq \delta \right\}$ be a disk that intersects $\Gz$ in a connected segment $B$ of $\Gz$ about $x_0$, that is, such that for some numbers $t_1$ and $s_1$ with $0<t_1<s_1<1$,
\begin{equation}
  B \;:=\; \disk \cap \Gz \;=\; \left\{\, \Gz(t) : t\in [0,t_1] \cup [s_1,1] \,\right\}.
\end{equation}
Denote the complementary connected component of $\Gz$ by
$A=\Gz[(t_1,s_1)]$, so that
\begin{equation}
  \Gz \,=\, A\, \mathring\cup\, B.
\end{equation}

(c) Let numbers $t_2$ and $s_2$ in $[0,1]$ such that $0<t_1<t_2<s_2<s_1<1$ be given, so that $\Gz(t_2)$ and $\Gz(s_2)$ lie in $A$.  Let $A'$ denote the subsegment of $A$ equal to $\Gz[(t_2,s_2)]$.

\smallskip
(d) A type \type perturbation $\Gamma$ of $\Gz$ is obtained by replacing $B$ by a simple Lipschitz perturbation curve $D$ which connects in a $C^2$ manner to the boundary points $\Gz(t_1)$ and $\Gz(s_1)$ of $A$ and which is otherwise contained in the interior of~$\Delta$.  $D$ is $C^{2,\alpha}$ except at one interior point $x'_0$ of $D$.  An open subset of $D$ containing $x'_0$ coincides with a translation-rotation of the intersection of a disk $\Delta'$ of radius $\delta'<\delta$ with a corner of a curve $\partial\Omega$ obtained from two intersecting circles of the same radius, oriented such that the exterior angle is equal to $2\theta$, as described in Fig.~\ref{fig:disks}.
\end{definition}

Lemma~\ref{lemma:resolvent} is the workhorse of the main theorem on eigenvalues in the essential spectrum (Theorem~\ref{thm:main}).
The type \type perturbations $\Gamma$ will be required to satisfy a certain Lipschitz condition that will ensure, according to Lemma~\ref{lemma:Lipschitz}, that $\SG$ is uniformly controlled in the $L^2(\Gamma)$ norm.  For the construction of such perturbations in Lemma~\ref{lemma:existence}, the Lipschitz constant $M$ will depend on the angle $\theta$ of the corner.

\begin{condition}[Lipschitz condition]\label{cond:Lipschitz}
Let $\Gz$ be a simple closed curve of class $C^2$ in $\RR^2$.
Let a triple $(U,\Delta_0,M)$ for $\Gamma_0$ be given, in which $\Delta_0$ is a closed disk contained in an open subset $\,U$ of $\,\RR^2$, such that $\Delta_0\cap\Gz$ is a simple curve of nonzero length, $M$ is a positive real number, and $U\cap\Gz$ is the graph of a function in some rotated coordinate system for $\RR^2$ with Lipschitz constant less than $M$.
A perturbation curve $\Gamma$ of $\Gz$ satisfies the {\em Lipschitz condition subject to the triple $(U,\Delta_0,M)$}, if the perturbation is confined to $\Delta_0$, that is,
$\Gz\setminus\Delta_0=\Gamma\setminus\Delta_0$, and $U\!\cap\Gamma$ is the graph of a function in some rotated coordinate system for $\RR^2$ with Lipschitz constant less than $M$.
\end{condition}

\begin{lemma}\label{lemma:Lipschitz}
Let $\Gz$ be a simple closed curve of class $C^2$ in $\RR^2$, and let $(U,\Delta_0,M)$ be a triple for $\Gz$ as described in Condition~\ref{cond:Lipschitz}.  There exists a constant $\CS>0$ such that, for each perturbation $\Gamma$ of $\Gz$ that satisfies the Lipschitz Condition~\ref{cond:Lipschitz} subject to the triple $(U,\Delta_0,M)$,
\begin{equation}
  \left| (\SG\psi,\psi)_{L^2(\Gamma)} \right| \,\leq\, \CS^2 (\psi,\psi)_{L^2(\Gamma)}
  \qquad \forall\psi\in L^2(\Gamma).
\end{equation}
\end{lemma}

\begin{proof}
In (\ref{eq:single}), we assume $\beta=1$; the proof is similar for general $\beta>0$.
It will first be proved that there exists a constant $C$ such that for every curve $\Gamma$ satisfying the conditions of Lemma~\ref{lemma:Lipschitz},
\begin{equation}\label{eq:intbd}
\sup_{x\in \Gamma} \frac{1}{2\pi} \int_{\Gamma}\big| \log|x-y| \big| d\sigma_y  \;\leq\; C. 
\end{equation} 
Suppose $\Gamma$ is any such curve.  The constant $C$ obtained by the following analysis will not depend on the particular choice of $\Gamma$.

The conditions in Lemma \ref{lemma:Lipschitz} guarantee that there exists a collection $\{U^i\}_{i=0}^N$ of open subsets of $\RR^2$, independent of $\Gamma$, and rotated coordinate systems $\{(\xi^i,\eta^i)\}_{i=0}^N$ for $\RR^2$ such that
$U^0=U$,
$U^i \cap \Delta_0 = \emptyset$ for $i=1,\dots,N$,
$\{U^i\}_{i=0}^N$ covers $\Gamma$,
and for $i=0,\dots,N$, the intersection $U^i\!\cap\Gamma$ is the graph $\eta^i=f^i(\xi^i)$ of a Lipschitz function $f^i$ on an interval $(\xi^i_1,\xi^i_2)$.  The collection $\{U^i\}_{i=1}^N$ can be taken to be fine enough so that all the functions $f^i$ have Lipschitz constant bounded by $M$.

Denote the arclength of any curve $\gamma$ by $\meas(\gamma)$.
For the cover $\{U^i\}_{i=0}^N$, there exists a number $r_0: 0<r_0< 1$, such that for every $x\in \Gamma$, there exists an integer $i_x:0 \leq i_x \leq N$ and a segment $\gamma_x$ of $\Gamma$ such that $x\in\gamma_x \subset U^{i_x}$ and $\meas(\gamma_x)=2r_0$, with $x$ located at the center of $\gamma_x$ with respect to arclength.
Inside the chart $U^{i_x}$, $\gamma_x$ is parameterized by $\eta_{i_x} = f^{i_x}(\xi^{i_x})$ for $\xi^{i_x}\in(\xi^{i_x}_1,\xi^{i_x}_2)$.  With $x$ equal to the point $(\xi^{i_x}_0,f^{i_x}(\xi^{i_x}_0))$, it follows that $|\xi^{i_x}_1-\xi^{i_x}_0|\leq r_0$ and $|\xi^{i_x}_2-\xi^{i_x}_0|\leq r_0$.
The number $r_0$ can be taken to be independent of the choice of $\Gamma$ satisfying the Lipschitz Condition~\ref{cond:Lipschitz} subject to the triple $(U,\Delta_0,M)$ because $\Gamma$ differs from $\Gz$ only within the disk $\Delta_0$.

The integral in \eqref{eq:intbd} can be split into two parts,
\begin{equation}\label{split}
\int_{\Gamma} \big| \log|x-y| \big| d\sigma_y \;=\; \int_{\gamma_x} \big| \log|x-y| \big| d\sigma_y \,+\,  \int_{\Gamma\setminus\gamma_x} \big| \log|x-y| \big| d\sigma_y.
\end{equation} 
The first term is bounded by
\begin{align}
\int_{\gamma_x} \big| \log|x-y| \big| d\sigma_y
&\;=\; \int_{ (\xi^{i_x}_1, \xi^{i_x}_2)} \left| \log \sqrt{(\xi^{i_x}_0 - \xi^{i_x})^2 +(f^{i_x}(\xi^{i_x}_0) - f^{i_x}(\xi^{i_x}))^2\,}\right| d\sigma(\xi^{i_x}) \label{hi} \\
&\;\leq\; \int_{ (\xi^i_{x1}, \xi^i_{x2})}  \left| \log | \xi^{i_x}_0 - \xi^{i_x} | \right| \sqrt{M^2 +1}\, d\xi^{i_x} \\
&\;\leq\; \int_{ (-r_0, r_0)} \big| \log |r| \big| \sqrt{M^2 +1}\, dr 
\;=\; C',
\end{align} 
where $C'$ is a finite constant.  This constant depends only on $r_0$ and $M$ and is therefore independent of the choice of $\Gamma$ satisfying the Lipschitz Condition~\ref{cond:Lipschitz} subject to the triple $(U,\Delta_0,M)$.  The first inequality comes from $r_0<1$, which makes the argument of the logarithm of~(\ref{hi}) less than~$1$.
The second term of~(\ref{split}) is bounded~by
\begin{equation}\label{bound2}
\int_{\Gamma\setminus\gamma_x} \big| \log|x-y| \big| d\sigma_y \;\leq\; \meas(\Gamma)\max\left(\big|\log|r_1(\Gamma)| \big|, \big|\log|r_2(\Gamma)| \big|\right),
\end{equation} 
where $r_1(\Gamma): = \inf_{x\in \Gamma} \text{dist}(x, \Gamma\setminus\gamma_x)$ and $r_2(\Gamma)$ is the radius of $\Gamma$. For $\Gamma$ satisfying the Lipschitz Condition~\ref{cond:Lipschitz} subject to the triple $(U,\Delta_0,M)$, $\meas(\Gamma)$ is uniformly bounded from above and both $r_1(\Gamma)$ and $r_2(\Gamma)$ are uniformly bounded from above and below by positive numbers.  Therefore, the right-hand side of (\ref{bound2}) is bounded by a constant $C''$ that does not depend on this choice of $\Gamma$.

With the constant $C=(C'+C'')/(2\pi)$, the bound $\eqref{eq:intbd}$ is proved for all curves $\Gamma$ satisfying the Lipschitz Condition~\ref{cond:Lipschitz} subject to the triple $(U,\Delta_0,M)$.
By Young's generalized inequality~\cite[Theorem~0.10]{Folland1995}, \eqref{eq:intbd} implies that
\begin{equation}
\left\| \SG\psi \right\|_{L_2(\Gamma)} \leq C \left\| \psi \right\|_{L_2(\Gamma)}
\end{equation} 
for all such curves $\Gamma$.  Thus the conclusion of Lemma \ref{lemma:Lipschitz} holds for $\CS = \sqrt{C}$.
\end{proof}

For the proof of Lemma \ref{lemma:resolvent}, we will work within the spaces $\Hmhalfz(\Gamma)$ to ensure that $\langle \cdot,\,\cdot \rangle_{\SG}$ remains positive definite.  
In $\Hmhalfz(\Gamma)$, the $\mathcal{S}$ inner product is independent of the choice of $\beta>0$ in the single-layer potential operator~(\ref{eq:single}).  We set $\beta=1$.

\begin{lemma}\label{lemma:resolvent}
  Let a simple closed curve $\Gz$ of class $C^{2,\alpha}$ ($\alpha>0$) in $\RR^2$, an eigenvalue $\lambda\not\in\left\{ 0, \frac{1}{2} \right\}$ of $\KsGz$, and a number $\epsilon>0$ be given;
 and let a triple $(U,\Delta_0,M)$ for $\Gz$ be given as in Condition~\ref{cond:Lipschitz}.

(1) There exist numbers $r>0$ and $\rho>0$ such that, 
for each type \type perturbation $\Gamma$ of $\Gz$ that satisfies the Lipschitz Condition~\ref{cond:Lipschitz} subject to $(U,\Delta_0,M)$, and the condition
\begin{equation}\label{condition1}
 0<t_2<r, \qquad 0<1-s_2<r, 
\end{equation}
and the condition
\begin{equation}\label{condition2}
  \frac{\sqrt{\meas(D)\,}}{\dist(A', D)} \;<\; \rho\,,
\end{equation}
($\meas(D)$ is the arclength of the curve $D$), there exists $\psi \in \Hmhalfz(\Gamma)$ satisfying
\begin{equation}\label{Sbound0}
\langle (\mathcal K_{\Gamma}^* - \lambda )\psi, (\mathcal K_{\Gamma}^* - \lambda )\psi \rangle_{\SG}
 \;\leq\; \epsilon^2\, \langle \psi, \psi  \rangle_{\SG}\,.
\end{equation}
Thus, either $\lambda\in\sigma(\KsG)$ or
\begin{equation}\label{eq:resolventlowerbound}
  \left\| (\KsG-\lambda)^{-1} \right\|_{\SG} > \epsilon^{-1}
\end{equation}
where $\KsG$ is considered as an operator in $\Hmhalfz(\Gamma)$.

(2) If $\Gz$ has reflectional symmetry about a line $L$ and $\Delta_0$ contains an intersection point of $L$ and $\Gamma_0$ and $\lambda$ is an eigenvalue of the even component $\KsGze$ of $\KsGz$ (or the odd component $\KsGzo$), then~(\ref{eq:resolventlowerbound}) can be replaced~by 
\begin{equation}\label{eq:resolventlowerboundeven}
  \left\| (\KsGe-\lambda)^{-1} \right\|_{\SG} > \epsilon^{-1}
  \qquad\big(\text{or}\;\;
  \left\| (\KsGo-\lambda)^{-1} \right\|_{\SG} > \epsilon^{-1} \,\big) 
\end{equation}
(considered as an operator in the even (odd) subspace of $\Hmhalfz(\Gamma)$)
for each type \type perturbation $\Gamma$ of $\Gz$ that has reflectional symmetry about $L$, satisfies the Lipschitz Condition~\ref{cond:Lipschitz} subject to $(U,\Delta_0,M)$, and satisfies~(\ref{condition1}) and~(\ref{condition2}).
\end{lemma}

\begin{proof}
Let $\lambda \notin \left\{ 0, \frac{1}{2} \right\}$ be an eigenvalue of $\KsGz:\Hmhalf(\Gz)\to\Hmhalf(\Gz)$ with eigenfunction $\phi$,
\begin{equation}
(\KsGz - \lambda )\phi = 0\,.
\end{equation}
We may assume that $\phi$ is real-valued since the kernel of $\KsGz$ is real.
By Lemma~\ref{lemma:L2}, $\phi \in L^2(\Gz)$.  By Theorem~3.6 of~\cite{ColtonKress1998}, $\KsGz$ maps $L^2(\Gz)$ into $H^1(\Gz)$ because $\Gz$ is of class $C^{2,\alpha}$, thus $\phi$ is an absolutely continuous function (in the almost-everywhere sense).
Since $\phi$ is not in the one-dimensional eigenspace for the eigenvalue $1/2$ of $\KsGz$, it must lie in the $\SGz$-complement $\Hmhalfz(\Gz)$ of that eigenspace, that is, $\phi\in\Hmhalfz(\Gz)$.
Recall that $\langle \cdot,\,\cdot \rangle_{\SGz}$ is positive definite in $\phi\in\Hmhalfz(\Gz)$ and the corresponding norm is denoted by $\|\cdot\|_{\SGz}$.

Let $(U,\Delta_0,M)$ be a triple for $\Gz$ as in Condition~\ref{cond:Lipschitz}, and let $\CS$ be the constant provided by Lemma~\ref{lemma:Lipschitz}.
Let $\Gamma$ be a type \type perturbation of $\Gz$, with all notation from Definition~\ref{def:Tperturbation} pertaining to it, that satisfies the Lipschitz Condition~\ref{cond:Lipschitz} subject to $(U,\Delta_0,M)$.
By Lemma~\ref{lemma:Lipschitz},
\begin{equation}\label{uniformbound}
  \left\| \psi \right\|_{\SG}^2 \;=\; (\psi,\psi)_{\SG} \;:=\; (\SG\psi,\psi)_{L^2(\Gamma)} \,\leq\, \CS^2 (\psi,\psi)_{L^2(\Gamma)}
  \qquad \forall\psi\in L_0^2(\Gamma),
\end{equation}
in which $L_0^2(\Gamma)$ denotes the space of all $f\in L^2(\Gamma)$ such that $\int_\Gamma f ds=0$.
This uniform bound will not be used until inequality~(\ref{Sbound}).

Let $x_1 \in \Gz$ be a point other than $x_0$, such that $|\phi(x_1)|>\frac{3}{4}\max_{y\in\Gz}|\phi(y)|$. Let $J$ be a subarc of $\Gz$ containing $x_1$. There exists a number $d>0$, such that when $\meas(J)<d$, $\phi$ does not change sign on $J$, $|\phi(x)| > \frac{1}{2}\max_{y\in\Gz}|\phi(y)|$ for $x\in J$ and $J\subset A'$ when $\meas(\Gz \backslash A') < d$. 
For every choice of $t_2$ and $s_2$ such that $\meas(\Gz \backslash A') < d$, let $\meas(J) = \meas (\Gz \backslash A')$.  Then one can choose constant $a : -2<a<2$ in the function
\begin{equation}
\chi(x) = 
\begin{cases}
1,\quad x\in A'\backslash J\\
a, \quad x\in J\\
0, \quad \text{otherwise}
\end{cases}
\end{equation}
such that $\chi\phi \in L_0^2(\Gz)\subset \Hmhalfz(\Gz)$.  Since $\chi\phi$ is supported in $A'$, which is a subarc of both $\Gamma$ and $\Gz$, $\chi\phi$ can also be considered to lie in $\Hmhalfz(\Gamma)$.

Let $C_0$ be a bound for $\SGz$ in $L^2(\Gz)$,
\begin{equation}
  \left\| \SGz\psi \right\|_{L^2(\Gz)} \;\leq\; C_0 \left\| \psi \right\|_{L^2(\Gz)}
  \qquad \forall\psi\in L^2(\Gz).
\end{equation}
Thus
\begin{equation}
\begin{split}
  \left| \left( \chi\phi, \chi\phi \right)_{\SG} - \left( \phi, \phi \right)_{\SGz} \right|
  &\;=\; \left| \left( \chi\phi, \chi\phi \right)_{\SGz} - \left( \phi, \phi \right)_{\SGz} \right|
    \;=\; \left| \left( \chi\phi-\phi, \chi\phi \right)_{\SGz} + \left( \phi, \chi\phi-\phi \right)_{\SGz} \right| \\
  &\;\leq\; C_0 \left( \left\| \chi\phi \right\|_{L^2(\Gz)} + \left\| \phi \right\|_{L^2(\Gz)} \right) \left\| \chi\phi-\phi \right\|_{L^2(\Gz)} \\
  &\;\leq\; 3C_0\left\| \phi \right\|_{L^2(\Gz)} \left\| (1-\chi)\phi \right\|_{L^2(\Gz)}.
\end{split}
\end{equation}
As $t_2$ and $1-s_2$ tend to zero simultaneously, the measure of the support of $1-\chi$ on $\Gz$ tends to zero, and therefore $\left\| (1-\chi)\phi \right\|_{L^2(\Gz)}$ converges to zero.
Thus, $\left( \chi\phi, \chi\phi \right)_{\SG}$ converges to $\left( \phi, \phi \right)_{\SGz}$; equivalently, $\|\chi\phi\|_{\SG}$ converges to $\|\phi\|_{\SGz}$ as $t_2$ and $1-s_2$ tend to zero.
The number $\Cphi:=\|\phi\|_{\SGz}/2$\, is positive because $\SGz$ is a positive operator and $\phi$ is nonzero in $L_0^2(\Gz)$.  Therefore,
\begin{equation}\label{Cphi}
  \left\| \chi\phi \right\|_{\SG} \;>\; \Cphi 
\end{equation}
whenever $t_2$ and $1-s_2$ are sufficiently small.

We next seek to bound the $L^2$ norm $\|(\mathcal K_{\Gamma}^* - \lambda )(\chi \phi) \|_{L^2(\Gamma)}$ (see (\ref{mainbound}) below).
The domains $A$ and $D$ can be treated separately since
\begin{equation}
\|(\mathcal K_{\Gamma}^* - \lambda )(\chi \phi) \|_{L^2(\Gamma)} \,\leq\, \|(\mathcal K_{\Gamma}^* - \lambda )(\chi \phi) \|_{L^2(A)} + \|(\mathcal K_{\Gamma}^* - \lambda )(\chi \phi) \|_{L^2(D)}\,.
\end{equation}
For the set $A$, one uses the eigenvalue condition $(\KsGz - \lambda )\phi = 0$ and $\KsG(\chi \phi)|_A = \KsGz(\chi \phi)|_A $ to obtain
\begin{equation}\label{hello}
\begin{split}
[(\KsG - \lambda )(\chi \phi)]\big|_A 
&= [(\KsG - \lambda )(\chi \phi) - (\KsGz - \lambda )\phi ]\big|_A \\
&= [(\KsGz - \lambda )(\chi \phi) - (\KsGz - \lambda )\phi ]\big|_A \\
&= [\KsGz \left((\chi - 1) \phi\right) + \lambda (1 - \chi) \phi ]\big|_A\,.
\end{split}
\end{equation}
Denote the kernel of the adjoint Neumann-Poincar\'e operator by
\begin{equation}\label{kernel}
  \ksS(x,y) = \frac{1}{2\pi}\frac{x-y}{|x-y|^2}\cdot n_x
  \qquad
  (\Sigma = \Gz \text{ or } \Gamma).
\end{equation}
The first term in the last expression of (\ref{hello}) is bounded pointwise due to the pointwise bound
%
  $2\pi \ksGz(x,y)<C_{\Gz}$,  
%
which holds since $\Gz$ is of class $C^2$~\cite[Theorem~2.2]{ColtonKress1983},
\begin{equation}
\begin{split}
 2\pi |\KsGz((\chi-1)\phi)(x) |
 &= \left| \int_{\Gz} \ksGz(x,y) (\chi(y)-1) \phi(y)\, d\sigma(y) \right| \\
 &\leq 3\, C_{\Gz}  \int_{\Gz \!\setminus\! A' \cup J}   |\phi(y)|\, d\sigma(y) \\
 &\leq 3\, C_{\Gz} \| \phi \|_{L^2(\Gz)} \sqrt{\meas(\Gz \!\setminus\! A') + \meas(J)} \\
 &= 3\sqrt{2}\, C_{\Gz} \| \phi \|_{L^2(\Gz)} \sqrt{\meas(\Gz \!\setminus\! A')} \;,
 \qquad \forall\,x\in A,
\end{split}
\end{equation}
since $\meas(J) = \meas (\Gz \backslash A')$,
and the second term is bounded in norm by
\begin{equation}
 \| \lambda (1 - \chi) \phi\|_{L^2(A)} 
 \,\leq\, 3\, |\lambda| \left( \int_{\Gz\setminus A' \cup J} |\phi|^2 \right)^{1/2}
 \,\leq\, 3\, |\lambda|\, C(2\,\meas(\Gz\!\setminus\!A')),
\end{equation}
in which $C(\mu)>0$ is a number that decreases to zero as $\mu\to0$.
Together, these two bounds yield
\begin{equation}\label{eq:A}
\begin{split}
\|(\mathcal K_{\Gamma}^* - \lambda )(\chi \phi) \|_{L^2(A)} 
&\,\leq\, \|\mathcal K_{\Gz} ^* \left((\chi - 1) \phi\right)\|_{L^2(A)} + \| \lambda (1 - \chi) \phi\|_{L^2(A)}\\
&\,\leq\, \frac{ 3 C_{\Gz}}{\sqrt{2}\pi}  \| \phi \|_{L^2(\Gz)}  \sqrt{\meas(A)}\sqrt{\meas(\Gz\!\setminus\!A')} \,+\, 3\, |\lambda|\, C(2\,\meas(\Gz\!\setminus\!A')) \\
&\,\leq\, C'(\meas(\Gz\!\setminus\!A')) \,,
\end{split}
\end{equation}
in which $C'(\mu)>0$ is a number that decreases to zero as $\mu\to0$.

On the set $D$, $\chi\phi$ vanishes, so that
\begin{equation}\label{K-lambda}
(\mathcal K_{\Gamma}^* - \lambda )(\chi \phi) |_D = \mathcal K_{\Gamma}^* (\chi \phi) |_D.
\end{equation}
Since $\Gamma$ has a corner, the kernel of $\KsG$ does not enjoy a uniform pointwise bound, but \eqref{kernel} does provide
\begin{equation}\label{eq:bd}
|K_{\Gamma}^*(x,y)| \leq \frac{1}{2\pi} \frac{1}{|x-y|} \qquad \forall\, x, y\in \Gamma. 
\end{equation}
Using this and the inclusion $\supp(\chi)\subset A'$, one obtains a pointwise bound for $x\in D$,
\begin{equation}
\begin{split}
  \big| \KsG(\chi\phi)(x) \big|
  &= \left| \int_\Gamma \ksG(x,y)  \chi(y)\phi(y) d\sigma(y) \right|
  = 2 \left| \int_{A'} \ksGz(x,y) \phi(y) d\sigma(y) \right| \\
  &\leq\; \frac{1}{\pi\,\text{dist}(A', D)} \int_{A'} |\phi(y)| d\sigma(y) 
  \;\leq\; \frac{1}{\pi\,\text{dist}(A', D)} \| \phi \|_{L^2(\Gz)} \sqrt{\meas(\Gz)}
  \qquad \forall\,x\in D.
\end{split}
\end{equation}
This bound together with (\ref{K-lambda}) yields
\begin{align}\label{eq:D}
\|(\mathcal K_{\Gamma}^* - \lambda )(\chi \phi) \|_{L^2(D)}
\;\leq\; \frac{1}{\pi\,\text{dist}(A', D)} \| \phi \|_{L^2(\Gz)}  \sqrt{\meas(\Gz)\,\meas(D)}\,.
\end{align}
Combining \eqref{eq:A} and \eqref{eq:D} produces the bound
\begin{equation}\label{mainbound}
  \|(\mathcal K_{\Gamma}^* - \lambda )(\chi \phi) \|_{L^2(\Gamma)}
  \;\leq\; C'(\meas(\Gz\!\setminus\!A'))
  \,+ \left(\frac{ \|\phi\|_{L^2(\Gz)} \sqrt{\meas(\Gz)}}{\pi}  \frac{\sqrt{\meas(D)}}{\dist(A', D)}\right).
\end{equation}

Both of these bounding terms can be made arbitrarily small simultaneously.  Consider the first term:  $\Gz\!\setminus\!A'$ is the part of $\Gz$ about $x_0$ between $\Gz(t_2)$ and $\Gz(s_2)$.  Therefore, by taking $t_2$ and $1\!-\!s_2$ sufficiently small, $\meas(\Gz\!\setminus\!A')$ can be made arbitrarily small, and one obtains
\begin{equation}\label{len}
  C'(\meas(\Gz\!\setminus\!A')) \to 0
  \quad\text{as}\quad
  \max\left\{ t_2, 1\!-\!s_2 \right\} \to 0\,.
\end{equation}
Let $\epsilon>0$ be given arbitrarily.
The convergence~(\ref{len}) implies that there exists $r>0$ such that, if $0<t_2<r$ and $0<1-s_2<r$, then $C'(\meas(\Gz\!\setminus\!A'))<\epsilon\,\Cphi/(2\CS)$.
Assume that $r$ is small enough so that also (\ref{Cphi}) holds.
Then with $\rho=\epsilon\pi\Cphi/(2\CS\|\phi\|_{L^2(\Gz)} \sqrt{\meas(\Gz)}\,)$, the second term of (\ref{mainbound}) is less than $\epsilon\,\Cphi/(2\CS)$ whenever $\sqrt{\meas(D)}/\dist(A', D)<\rho$.  
These two bounds together yield
\begin{equation}
  \|(\KsG - \lambda )(\chi \phi) \|_{L^2(\Gamma)}
  \;\leq\; \frac{\Cphi}{\CS}\,\epsilon\,.
\end{equation}

Combining this bound with (\ref{uniformbound}) and (\ref{Cphi}) provides the desired bound
\begin{equation}\label{Sbound}
\|(\KsG - \lambda )(\chi \phi) \|_{\SG}
 \;\leq\; \CS \|(\mathcal K_{\Gamma}^* - \lambda )(\chi \phi) \|_{L^2(\Gamma)}
 \;\leq\; \epsilon\,\Cphi
 \;\leq\; \epsilon\,\|\chi\phi\|_{\SG}.
\end{equation}
If $\lambda \notin \sigma(\mathcal K_{\Gamma}^*)$ is a regular point of the operator $\mathcal K_{\Gamma}^*$,
this implies that 
\begin{equation}
\|(\KsG - \lambda )^{-1} \|_{\SG} > \epsilon^{-1},
\end{equation}
in which $\KsG$ is considered as an operator in $\Hmhalfz(\Gamma)$,
as claimed in the first part of the theorem.

These arguments also prove the second part of the theorem for a curve $\Gz$ that is symmetric about a line~$L$ if (1) the reference point $x_0$ is taken to be on $L$, (2) $J$ consists of two segments that are symmetric about $L$, (3) one takes $\chi$ to be even ($\Gz(s_2)$ is the reflection of $\Gz(t_0)$ about~$L$) so that if $\phi$ is even (or odd) $\chi\phi$ will also be even (or odd), and (4) the replacement curve $D$ is taken to be symmetric about~$L$.  Then in every occurrence of $\KsGz$ or $\KsG$ in the arguments, the operator is acting on an even (or odd) distribution, and thus may be replaced by $\KsGze$ or $\KsGe$ (or $\KsGzo$ or $\KsGo$).
\end{proof}

It is geometrically straightforward, even if somewhat technical analytically, to demonstrate that Lipschitz perturbations of type \type as required in Lemma~\ref{lemma:resolvent} are plentiful.  The following lemma will suffice.  Essentially, it says that one can always construct a perturbation $\Gamma$ with a desired corner angle $\theta$ for which the lower bound on the resolvent of $\KsG$ in Lemma~\ref{lemma:resolvent} holds.  To do this, one must find an appropriate Lipschitz constant $M$ for the given $\theta$ (sharper angles require larger $M$) and then construct a type \type perturbation that satisfies the requirements of Lemma~\ref{lemma:resolvent}.

\begin{lemma}\label{lemma:existence}
  Let a simple closed curve $\Gz$ of class $C^2$, and a number $\theta$ such that $0<\theta<\pi$ be given.  There exists a triple $(U,\Delta_0,M)$ for $\Gz$ as in Condition~\ref{cond:Lipschitz} such that, for all positive numbers $r$ and $\rho$, there exists a perturbation $\Gamma$ of $\Gz$ of type \type such that: $\Gamma$ satisfies the Lipschitz Condition~\ref{cond:Lipschitz} subject to $(U,\Delta_0,M)$;  conditions (\ref{condition1}) and~(\ref{condition2}) of Lemma~\ref{lemma:resolvent} are satisfied; and the half exterior angle of the corner of $\Gamma$ is equal to $\theta$.  If\, $\Gz$ is symmetric about a line~$L$, then $\Gamma$ can be taken to be symmetric about $L$ with the tip of the corner lying on $L$.
\end{lemma}

\begin{proof}
Given $\theta\in(0,\pi)$, let $g(\xi)$, for $\xi$ in some interval, be a function whose graph describes a rotated corner of a type \type perturbation as described in part (d) of Definition~\ref{def:Tperturbation} (a neighborhood of a corner of the intersection of two circles as in Fig.~\ref{fig:disks}) such that the tip occurs at $\xi=0$ and points upward for $\theta>\pi/2$ and downward for $\theta<\pi/2$; and let $M_1$ and $M_2$ be positive numbers such that $M_1<|g'(\xi)|<M_2$ for $\xi\not=0$.

Let a simple closed curve $\Gz$ of class $C^2$ be parameterized such that $\Gz(0)=\Gz(1)=x_0$.  Choose an open set $U\subset\RR^2$ and rotated and translated coordinates $(\xi,\eta)$ for $\RR^2$ such that $x_0\in U$ and $\Gz\cap U$ is the graph $\eta=f(\xi)$ of a $C^2$ function $f$, with $x_0=(0,f(0))$ and $|f'(\xi)|<\min\{1,M_1\}$, and such that the part of $U$ that lies below the graph is in the interior domain of $\Gz$.
Choose a closed disk $\Delta_0\subset U$ centered at $x_0$.
Each closed circle centered at $x_0$ contained in $\Delta_0$ intersects $\Gz$ at exactly two points.  There are no more than two intersection points because $|f'(\xi)|<1$.

Let $\Delta$ be any closed disk centered at $x_0$ and contained in $\Delta_0$.
Define $\tilde g(\xi)=g(\xi)+\eta_0$ with $\eta_0$ chosen such that the graph $\eta=\tilde g(\xi)$ intersects $\Gz$ in exactly two points in the interior of $\Delta$---call them $x_1=(\xi_1,f(\xi_1))$ and $x_2=(\xi_2,f(\xi_2))$---and such that the graph of $\tilde g$ between these two points lies in the interior of $\Delta$.  This is possible because $|\tilde g'(\xi)|>M_1$ and $|f'(\xi)|<M_1$.

Set $\tilde f(\xi)=f(\xi)$ for $\xi\not\in[\xi_1,\xi_2]$ and $\tilde f(\xi)=\tilde g(\xi)$ for $\xi\in[\xi_1,\xi_2]$, and observe that the tip of the corner occurs at the point $(0,\tilde f(0))$.  Then let $\tilde{\tilde f}(\xi)$ be a function that is of class $C^2$ except at $\xi=0$ and that is equal to $\tilde f(\xi)$ except in two nonintersecting intervals, one about $\xi_1$ and one about $\xi_2$; these intervals can be taken small enough so that the graphs of $\tilde{\tilde f}$ and $f$ coincide outside of $\Delta$.  The smoothing $\tilde{\tilde f}$ can also be arranged so that $\big|\tilde{\tilde f}'(\xi)\big|<M_2$; this is because $|\tilde f'(\xi)|<M_2$ except at $\xi_1$, $0$, and $\xi_2$, where $\tilde f$ is continuous but not differentiable.  It follows that the length of the graph of $\tilde{\tilde f}$ inside $\Delta$, which is called $D$ in part (d) of Definition~\ref{def:Tperturbation}, is bounded by 
%
\begin{equation}
  \meas(D) \leq 2\sqrt{1+M_2^2\,}\,\mathrm{rad}(\Delta).
\end{equation}
The curve $\Gamma$ resulting from replacing the segment of $\Gz$ described by $\eta=f(\xi)$ by the curve $\eta=\tilde{\tilde f}(\xi)$ is a type~\type perturbation of $\Gz$ that satisfies the Lipschitz Condition~\ref{cond:Lipschitz} subject to the triple $(U,\Delta_0,M_2)$, and its corner has half exterior angle equal to $\theta$.

Let $r>0$ and $\rho>0$ be given.  Choose numbers $t_2$ and $s_2$ in Definition~\ref{def:Tperturbation} so that condition~(\ref{condition1}) is satisfied, that is, $0<t_2<r$ and $0<1-s_2<r$.
For these fixed values of $t_2$ and $s_2$,
\begin{equation}
  \frac{\sqrt{\meas(D)}}{\dist(A',D)} \leq \frac{\sqrt{2\sqrt{1+M_2^2\,}\,\mathrm{rad}(\Delta)}}{\dist(A',\Delta)} \to 0
  \qquad
  \text{as } \; \mathrm{rad}(\Delta)\to0.
\end{equation}
Therefore, $\mathrm{rad}(\Delta)$ can be taken to be small enough in this construction of $\Gamma$ so that
\begin{equation}
  \frac{\sqrt{\meas(D)}}{\dist(A',D)} < \rho,
\end{equation}
which is condition~(\ref{condition2}).
In the symmetric case, $x_0\in L$ and $t_2$ and $s_2$ can be chosen such that $\Gz(s_2)$ is the reflection of $\Gz(t_0)$ about $L$, and $D$ can be arranged to be symmetric about~$L$.
\end{proof}

\section{Reflection symmetry and essential spectrum}\label{sec:essential}

For all of the curves in this section, assume that $\beta$ in (\ref{eq:single}) is chosen such that $\mathcal{S}$ is a positive operator for all the curves under consideration.
Consider a curve $\Gz$ of class $C^2$ and perturbations $\Gamma$ of type \type that are symmetric with respect to a line $L$.  Recall that, in this case, the operators $\KsGz$ and $\KsG$ admit decompositions onto the even and odd distributional spaces, as stated in~(\ref{eq:Kdecomposition}),
\begin{equation}
  \KsGz = \KsGze\oplus\KsGzo\,,
  \quad
  \KsG = \KsGe\oplus\KsGo\,.
\end{equation}
The prototypical curvilinear polygons $\partial\Omega$ described in section~\ref{sec:perturbation} (Fig.~\ref{fig:disks}) are themselves symmetric about a line through the two corner points.  The spectral resolution of the Neumann-Poincar\'e operator on $\partial \Omega$ is explicitly computed in~\cite{KangLimYu2017} through conformal mapping and Fourier transformation.
Recall that $\theta$ is half the angle of the corner measured in the exterior of the curve.  It is shown that
\begin{equation}
\sigma_\acont(\KsO) = [-b,b],  \quad \sigma_\scont(\KsO) = \emptyset, \quad\sigma_\pp(\KsO) = \{1/2\},
\end{equation}
where $b = |\frac{1}{2} - \frac{\theta}{\pi}|$ depends on the angle, $\sigma_\acont$ refers to absolutely continuous spectrum, $\sigma_\scont$ refers to singular continuous spectrum, and $\sigma_\pp$ refers to pure point spectrum.   Therefore, $\sigma_\acont(\KsO)=\sigma_\ess(\KsO)$.

Furthermore, it is shown in~\cite{KangLimYu2017} that the essential spectra of the even and odd components of $\KsO$ intersect only in $\{0\}$,
 \begin{equation}
\sigma_\ess(\mathcal K_{\partial \Omega,o}^*) = [-b, 0],  \quad \sigma_\ess(\mathcal K_{\partial \Omega,e}^*) = [0, b] \quad \text{for} \quad \pi/2<\theta<\pi
\end{equation}
for outward-pointing corners and
 \begin{equation}
\sigma_\ess(\mathcal K_{\partial \Omega,o}^*) = [0,b],  \quad \sigma_\ess(\mathcal K_{\partial \Omega,e}^*) = [-b,0] \quad \text{for} \quad 0<\theta < \pi/2
\end{equation}
for inward-pointing corners.
Our proof of eigenvalues in the essential spectrum requires that this disjointness persist for the perturbation $\Gamma$, and this is the content of the following proposition.

The proof of Proposition~\ref{prop:evenodd} invokes the local nature of the essential spectrum of $\KsG$.  This is bridged by its essential spectrum $\sigma_{\text{ea}}(\KsG)$ in the approximate eigenvalue sense~\cite{PerfektPutinar2017}.  For an operator $T: H \rightarrow H$, $\lambda \in \sigma_{\text{ea}}(T)$ if and only if there is a bounded sequence $\{f_n\}_{n=1}^{\infty} \in H$ having no convergent subsequence, such that $(T - \lambda)f_n \rightarrow 0$ in $H$.  One calls $\{f_n\}_{n=1}^{\infty}$ a singular sequence.  When $T$ is self adjoint, $ \sigma_\ess (T) = \sigma_{\text{ea}}(T)$. If an operator $S: H \rightarrow H$ is such that $S-T$ is compact, then $ \sigma_{\text{ea}} (S) = \sigma_{\text{ea}}(T)$.

\begin{proposition}\label{prop:evenodd}
The essential spectra of the even and odd components of $\KsG$ for a reflectionally symmetric perturbation curve $\Gamma$ of type~\type coincides with the essential spectra of the even and odd components of $\KsO$ for the prototypical curvilinear polygon $\partial\Omega$ (Fig.~\ref{fig:disks}) having corners with the same exterior angle as $\Gamma$,
\begin{eqnarray}
    && \sigma_\ess(\KsGe) \;=\; \sigma_\ess(\KsOe), \\
    && \sigma_\ess(\KsGo) \;=\; \sigma_\ess(\KsOo).
\end{eqnarray}
\end{proposition}

\begin{proof}
This proof essentially follows~\cite{PerfektPutinar2017}.
Let $\Sigma$ be a simple closed Lipschitz curve that is piecewise of class $C^2$ and has $n$ corners.
Let $\{\rho_j\}_{j=1}^n$ be cutoff functions on $\Sigma$ that have mutually disjoint supports and such that $\rho_j$ is equal to $1$ in a neighborhood of the $j$-th corner and is of class $C^2$ otherwise, and set $\rho_0 = 1-\sum_{j=1}^n \rho_j$.
Denote by $M_{\rho}$ the operator of multiplication by $\rho$.  In the decomposition
\begin{equation}
  \KS \;=\;  \sum\limits_{0\leq i,j \leq n} M_{\rho_i} \KS M_{\rho_j},
\end{equation}
each term is compact unless $i=j\not=0$.  This implies the second equality in
\begin{equation}\label{ess}
  \sigma_\ess(\KS) \;=  \sigma_{\text{ea}}(\KS) = \; \sigma_{\text{ea}} \left(\textstyle\sum\limits_{j=1}^n  M_{\rho_j} \KS M_{\rho_j} \right)
  = \; \textstyle\bigcup\limits_{j=1}^n \sigma_{\text{ea}} \left( M_{\rho_j} \KS M_{\rho_j} \right),
\end{equation}
where the first equality follows from the self-adjointness of $\KS: \Hhalf(\Sigma)\to\Hhalf(\Sigma)$ with respect to the $\SG^{-1}$ inner product, and the last equality is proved in \cite[Lemma 9]{PerfektPutinar2017}.

Now suppose that $\Sigma$ is reflectionally symmetric about a line $L$ and that $\Sigma$ has either one or two corners (so that $n=1$ or $n=2$) with vertex on $L$ and that the cutoff functions $\rho_j$ are chosen to be even so that the operators $M_{\rho_j}$ commute with the reflection.  Because of this, one has orthogonal decompositions
\begin{equation}
  M_{\rho_i} \KS M_{\rho_j} \;=\;
  M_{\rho_i} \KSe M_{\rho_j} \,\oplus\, M_{\rho_i} \KSo M_{\rho_j},
\end{equation}
and therefore the compactness of $M_{\rho_i} \KS M_{\rho_j}$ (unless $i=j\not=0$) implies the compactness of the even and odd components on the right-hand side.  Using this with the decomposition
\begin{equation}
  \KSe \;=\; \sum\limits_{0\leq i,j \leq n} M_{\rho_i} \KSe M_{\rho_j}
\end{equation}
and the analogous decomposition of $\KSo$ yields
\begin{eqnarray}
  && \sigma_\ess(\KSe) \;=\; \textstyle\bigcup\limits_{j=1}^n \sigma_{\text{ea}} \left( M_{\rho_j} \KSe M_{\rho_j} \right), \\
  && \sigma_\ess(\KSo) \;=\; \textstyle\bigcup\limits_{j=1}^n \sigma_{\text{ea}}\left( M_{\rho_j} \KSo M_{\rho_j} \right).\end{eqnarray}

Apply this result to $\partial\Omega$, which has two corners ($n=2$), and to the type \type perturbation $\Gamma$ of $\Gz$, which has only one corner ($n=1$), and use $\tilde\rho_1$ for $\Gamma$ to distinguish it from $\rho_1$ for $\partial\Omega$,
\begin{equation}
\begin{aligned}
  \sigma_\ess(\KOe) &\;=\; \sigma_{\text{ea}} \left( M_{\rho_1} \KOe M_{\rho_1} \right) 
                   \cup \sigma_{\text{ea}} \left( M_{\rho_2} \KOe M_{\rho_2} \right), \\
  \sigma_\ess(\KGe) &\;=\; \sigma_{\text{ea}} \left( M_{\tilde\rho_1} \KGe M_{\tilde\rho_1} \right).
\end{aligned}
\end{equation}
Since a neighborhood of the corner of $\Gamma$ coincides after translation and rotation with a neighborhood of either corner of $\partial\Omega$, and since $\partial\Omega$ has symmetry about a vertical line (Fig.~\ref{fig:disks}), the function $\rho_1+\rho_2$ can be chosen to be symmetric with respect to both reflections.  Furthermore, $\tilde\rho_1$ and $\rho_1$ can be chosen so that $\supp\,\tilde\rho_1 \cap \Gamma$ and $\supp\,\rho_1 \cap \partial\Omega$ as well as the functions $\tilde\rho_1$ and $\rho_1$ on their supports coincide after translation and rotation.
Under these conditions, $M_{\rho_1} \KOe M_{\rho_1}$, and $M_{\rho_2} \KOe M_{\rho_2}$ are unitarily similar operators, thus
\begin{eqnarray}
  \sigma_{\text{ea}} \left( M_{\rho_1} \KOe M_{\rho_1} \right) = \sigma_{\text{ea}}\left( M_{\rho_2} \KOe M_{\rho_2} \right) =  \sigma_\ess(\KOe).
\end{eqnarray}
Since $\sigma_{\text{ea}} \left( M_{\rho_j} \KSe M_{\rho_j} \right)$ is characterized by functions localized at the $j$-th corner, we obtain
\begin{eqnarray}
   \sigma_{\text{ea}} \left( M_{\tilde\rho_1} \KGe M_{\tilde\rho_1} \right) = \sigma_{\text{ea}}\left( M_{\rho_1} \KOe M_{\rho_1} \right).
\end{eqnarray}
Therefore
\begin{eqnarray}
    && \sigma_\ess(\KGe) \;=\; \sigma_\ess(\KOe), \\
    && \sigma_\ess(\KGo) \;=\; \sigma_\ess(\KOo),
\end{eqnarray}
and the equation for the odd component is obtained in the same manner.

The proposition now follows from $\sigma_\ess(\KsGe) = \sigma_\ess(\KGe)$ and 
$\sigma_\ess(\KsOe) = \sigma_\ess(\KOe)$ and the analogous equalities for the odd components of these operators, where $\KsG$ and $\KsO$ are considered on $\Hmhalf(\Gamma)$ and $\Hmhalf(\partial\Omega)$.
\end{proof}

Equation (\ref{ess}) expresses the local manner in which the corners of a curvilinear polygon contribute to the essential spectrum of the Neumann-Poincar\'e operator.  How this happens for an individual corner is enlightened through explicit construction of Weyl sequences associated to each $\lambda\in\sigma_\ess(\KsG)$, which is carried out by Bonnetier and Zhang~\cite{BonnetierZhang2017}.

\section{Eigenvalues in the essential spectrum}\label{sec:embedded}

The strategy to construct eigenvalues in the essential spectrum for the Neumann-Poincar\'e operator is to obtain a spectral-vicinity result of the form
\begin{equation}\label{eq:dist}
  \dist(\lambda,\sigma(\KsGe)) < \epsilon\,,
\end{equation}
in which $\lambda$ is an eigenvalue of $\KsGze$ and $\Gamma$ is a type \type perturbation of $\Gz$, by applying Lemma~\ref{lemma:resolvent}.  The angle of the corner of $\Gamma$ is chosen so that $\lambda$ does not lie within the essential spectrum of $\KsGe$ but does lie inside the essential spectrum of $\KsGo$.  This will guarantee that $\KsGe$ has an eigenvalue near $\lambda$ and that this eigenvalue lies in the essential spectrum of $\KsGo$.  An analogous procedure applies to eigenvalues of $\KsGzo$.  In fact, $\Gamma$ can be chosen so that several eigenvalues of $\KsGz$ are perturbed into eigenvalues of $\KsG$ that lie within the essential spectrum.  Our proof is only able to guarantee a finite number of eigenvalues in the essential spectrum for a given perturbation $\Gamma$.
This is because the perturbation $\Gamma$ depends on the eigenfunction and on $\epsilon$ (smaller $\epsilon$ requires a corner of smaller arclength), and no uniform $\epsilon$ can be chosen to guarantee infinitely many distinct perturbed eigenvalues of the same sign.

\begin{theorem}\label{thm:main}
Let $\Gz$ be a simple closed curve of class $C^{2,\alpha}$ in $\RR^2$ that is symmetric about a line $L$. 

\smallskip
\noindent
(a)  Suppose that the adjoint Neumann-Poincar\'e operator $\KsGz$ has $m$ even eigenfunctions corresponding to eigenvalues $\lambda^e_j$ and $n$ odd eigenfunctions corresponding to eigenvalues $\lambda^o_j$ such that
\begin{equation}
  \lambda^e_m<\dots<\lambda^e_1<0<\lambda^o_1<\dots<\lambda^o_n\,.
\end{equation}
There exists a Lipschitz-continuous perturbation $\Gamma$ of $\Gz$ with the following properties:  $\Gamma$ is symmetric about~$L$; $\Gamma$ possesses an outward-pointing corner and is otherwise of class $C^{2,\alpha}$; the associated operator $\KsG$ has $m$ even eigenfunctions corresponding to eigenvalues $\tilde\lambda^e_j$ and $n$ odd eigenfunctions corresponding to eigenvalues $\tilde\lambda^o_j$ such that
\begin{equation}\label{perturbedevals1}
  \tilde\lambda^e_m<\dots<\tilde\lambda^e_1<0<\tilde\lambda^o_1<\dots<\tilde\lambda^o_n\,;
\end{equation}
these eigenvalues lie within the essential spectrum of $\KsG$.

\smallskip
\noindent
(b)  Suppose that the adjoint Neumann-Poincar\'e operator $\KsGz$ has $m$ odd eigenfunctions corresponding to eigenvalues $\lambda^o_j$ and $n$ even eigenfunctions corresponding to eigenvalues $\lambda^e_j<1/2$ such that
\begin{equation}
  \lambda^o_m<\dots<\lambda^o_1<0<\lambda^e_1<\dots<\lambda^e_n\,.
\end{equation}
There exists a Lipschitz-continuous perturbation $\Gamma$ of $\Gz$ with the following properties:  $\Gamma$ is symmetric about~$L$; $\Gamma$ possesses an inward-pointing corner and is otherwise of class $C^{2,\alpha}$; the associated operator $\KsG$ has $m$ odd eigenfunctions corresponding to eigenvalues $\tilde\lambda^o_j$ and $n$ even eigenfunctions corresponding to eigenvalues $\tilde\lambda^e_j$ such that
\begin{equation}\label{perturbedevals2}
  \tilde\lambda^o_m<\dots<\tilde\lambda^o_1<0<\tilde\lambda^e_1<\dots<\tilde\lambda^e_n\,;
\end{equation}
these eigenvalues lie within the essential spectrum of $\KsG$.
\end{theorem}

\begin{proof}
For part (a), observe that $-\lambda^e_m$ and $\lambda^o_n$ are less than $1/2$ because $\sigma(\KsG)$ is contained in the interval $(-1/2,1/2)$ except for the eigenvalue $1/2$.  The eigenfunction for $1/2$ is even because it corresponds to the single-layer potential that is constant on $\Gamma$.
Choose a real number $b$ such that $-b<\lambda^e_m<\lambda^o_n<b<1/2$, and let 
$\theta$ be the number such that $b=\theta/\pi - 1/2$, so that $\pi/2<\theta<\pi$.
Let $\epsilon>0$ be given such that
\begin{equation}
  \epsilon < \min\left\{ \textstyle\half |\lambda^e_i - \lambda^e_{i+1} |,\, \half |\lambda^o_j - \lambda^o_{j+1} |,\,
  |\lambda^e_1| ,\,  |\lambda^o_1|,\, b-\lambda^o_n,\, b+\lambda^e_m \right\},
  \quad
  i=1,\dots,m-1,\;\; j=1,\dots,n-1.
\end{equation}

Let $(U,\Delta_0,M)$ be a triple for $\Gamma_0$ guaranteed by Lemma~\ref{lemma:existence} for the given value of $\theta$.  
For this triple $(U,\Delta_0,M)$ and $\epsilon$, let $r(\lambda)$ and $\rho(\lambda)$ be the numbers stipulated in Lemma~\ref{lemma:resolvent} for
$\lambda\in\{ \lambda^e_1,\dots,\lambda^e_m, \lambda^o_1,\dots,\lambda^o_n \}$, and let $r$ be the minimum of $r(\lambda)$ and $\rho$ be the minimum of $\rho(\lambda)$ over all these eigenvalues.   Lemma~\ref{lemma:existence} provides a perturbation~$\Gamma$ of type \type such that (i) $\Gamma$ satisfies the Lipschitz Condition~\ref{cond:Lipschitz} subject to the triple $(U,\Delta_0,M)$, (ii) its corner has exterior angle $2\theta$, (iii) the conditions (\ref{condition1}) and~(\ref{condition2}) of Lemma~\ref{lemma:resolvent} are satisfied, (iv) $\Gamma$ is symmetric about $L$.  For this Lipschitz curve $\Gamma$, Lemma~\ref{lemma:resolvent} guarantees that
\begin{equation}
  \|(\KsG - \lambda )^{-1} \|_{\SG} > \epsilon^{-1}
  \qquad
  \forall\,
  \lambda\in\{ \lambda^e_1,\dots,\lambda^e_m, \lambda^o_1,\dots,\lambda^o_n \},
\end{equation}
in which $\KsG$ is considered as an operator in $\Hmhalfz(\Gamma)$.
As $\KsG$ is self-adjoint in $\Hmhalfz(\Gamma)$ with respect to the $\SG$ inner product, one obtains
\begin{equation}
  \dist( \lambda,\, \sigma(\KsG)) < \epsilon
  \qquad
  \forall\,
  \lambda\in\{ \lambda^e_1,\dots,\lambda^e_m, \lambda^o_1,\dots,\lambda^o_n \}.
\end{equation}
Because of part (2) of Lemma~\ref{lemma:resolvent}, this inequality holds for the spectrum of the even and odd components of~$\KsG$,
\begin{align}
  \dist( \lambda^e_j,\, \sigma(\KsGe)) < \epsilon & \quad\text{for }\, j=1,\dots,m, \label{s1} \\
  \dist( \lambda^o_j,\, \sigma(\KsGo)) < \epsilon & \quad\text{for }\, j=1,\dots,n.  \label{s2}
\end{align}
By Proposition~\ref{prop:evenodd} and the discussion preceding it, the essential spectra of these operators are
\begin{eqnarray}
  && \sigma_\ess(\KsGe) = [0,b], \label{s3}\\
  && \sigma_\ess(\KsGo) = [-b,0], \label{s4}
\end{eqnarray}
with $b=\frac{\theta}{\pi}-\frac{1}{2}$.
Because of (\ref{s1},\ref{s3}), the choice of $\epsilon$, and the self-adjointness of $\KsGe$, there exist eigenvalues $\tilde\lambda^e_j$ for $j=1,\dots,m$ that satisfy~(\ref{perturbedevals1}).  Similarly, because of (\ref{s2},\ref{s4}), there exist eigenvalues $\tilde\lambda^o_j$ for $j=1,\dots,n$ that satisfy~(\ref{perturbedevals1}).  Because of the choices of $b$ and $\epsilon$, one has
\begin{eqnarray}
  && \tilde\lambda^e_j \in \sigma_\ess(\KsGo), \\
  && \tilde\lambda^o_j \in \sigma_\ess(\KsGe).
\end{eqnarray}
Since $\pi/2<\theta<\pi$, the corner is outward-pointing.

Part (b) is proven analogously.  In this case, $b=-\theta/\pi + 1/2$, so that $0<\theta<\pi/2$, and the corner is therefore inward-pointing.
\end{proof}

For any reflectionally symmetric curve of class $C^{2,\alpha}$ except for a circle, Theorem~\ref{thm:main} allows one to create lots of eigenvalues in the essential spectrum by appropriate Lipschitz perturbations.

\begin{corollary}
  Let $\Gz$ be a simple closed curve of class $C^{2,\alpha}$ in $\RR^2$ that is symmetric about a line $L$ but that is not a circle.  For any positive integer $n$, there exists a perturbation $\Gamma$ of type~\type\!\!, also symmetric about $L$, such that $\KsG$ admits $n$ negative and $n$ positive eigenvalues that lie within the essential spectrum of $\KsG$.
\end{corollary}

\begin{proof}
We begin with two facts.
(1) Except for when $\Gz$ is a circle, the operator $\KsGz$ is always of infinite rank~\cite[\S7.3--7.4]{Shapiro1992}.
(2) For each nonzero eigenvalue $\lambda$ of $\KsGz$ corresponding to an even (odd) eigenfunction, $-\lambda$ is an eigenvalue of $\KsGz$ corresponding to an odd (even) eigenfunction.  The symmetry of the point spectrum is proved in~\cite[Theorem~2.1]{HelsingKangLim2016}; and the statement about the parities of the eigenfunctions can be obtained from augmenting the proof of that theorem, using the assumption that the eigenfunction corresponding to $\lambda$ is even (odd).

Assume that $\Gz$ is not a circle.
Facts (1) and (2) together imply that both $\KsGzo$ and $\KsGze$ are of infinite rank.  This means that $\KsGzo$ has infinitely many negative eigenvalues or infinitely many positive eigenvalues.  Suppose the former case holds.  Then by (2), $\KsGze$ has infinitely many positive eigenvalues.  Thus, for any integer $n$, the hypotheses of part (b) of Theorem~\ref{thm:main} are satisfied.  In the other case, the hypotheses of part (a) are satisfied.  In either case, the conclusion of the corollary follows from the theorem.
\end{proof}

\noindent{\bfseries Example: A perturbed ellipse.}\hspace{0.5em}
Consider the Neumann-Poincar\'e operator for an ellipse, whose eigenvalues and eigenfunctions are known explicitly~\cite[\S3]{ChungKangKimLee2014}.  They take simple forms in the elliptic coordinates $(\varrho, \omega)$, which are related to the the Cartesian coordinates $x = (x_1, x_2)$ by
\begin{equation}
x_1 = R \cos \omega \cosh \varrho, \quad x_2 = R \sin \omega \sinh \varrho, \quad \varrho > 0, \quad 0 \leq \omega \leq 2\pi. 
\end{equation}
The set $E = \{ (\varrho, \omega) : \varrho = \varrho_0 \}$ is an ellipse with foci $( \pm R, 0)$. The non-one-half eigenvalues of the operator $\mathcal K_E^*$ are $\alpha_n$ and $-\alpha_n$ and the corresponding eigenfunctions are
\begin{equation}
  \phi_n^+ := \Xi(\varrho_0, \omega)^{-1}\cos n\omega,
  \qquad
  \phi_n^- := \Xi(\varrho_0, \omega)^{-1}\sin n\omega
     \qquad (n \geq 1),
\end{equation}
in which
\begin{equation}\label{alpha}
\quad \alpha_n = \frac{1}{2e^{2n\varrho_0}},
\qquad
\Xi (\varrho_0, \omega)
     = R \sqrt{\sinh^2 \varrho_0 + \sin^2 \omega\,}
     \qquad (n \geq 1).
\end{equation}

We make two observations. First, $\phi_n^{\pm}$ are in $L^2(E)$, as guaranteed by Lemma~\ref{lemma:L2}.
Second, $\phi_n^+$ are even about the major axis of the ellipse, $\phi_n^-$ are odd about the major axis, $\phi_{2k+1}^+$ and $\phi_{2k}^-$ are odd about the minor axis, and $\phi_{2k}^+$ and $\phi_{2k+1}^-$ are even about the minor axis. That is to say, all eigenfunctions corresponding to positive (negative) eigenvalues are even (odd) with respect to the major axis, and they alternate between odd and even with respect to the minor axis.

Let $L$ be the major axis of an ellipse $\Gz=E$.
The hypotheses of part (b) of Theorem~\ref{thm:main} are satisfied for any integers $m$ and $n$, and therefore one can perturb $\Gz$ to a domain $\Gamma$ by attaching an inward-pointing corner with its tip on $L$ (according to Definition~\ref{def:Tperturbation}) that is small enough so that $\KsG$ has eigenvalues within the essential spectrum as described in the conclusion of part~(b).
Now let $L$ be the minor axis of an ellipse $\Gz=E$.  Either of the hypotheses of parts (a) and (b) of the theorem can be satisfied for any $m$ and $n$, and thereby eigenvalues within the essential spectrum can be created for $\KsG$ according to the theorem.

\section{Discussion}\label{sec:discussion}

We end this article with some questions and observations.

\smallskip
{\bfseries 1.} Can $\KsG$ have infinitely many embedded eigenvalues, and might this actually occur typically?
Our proof guarantees only a finite number of eigenvalues within the essential spectrum for a given Lipschitz type \type perturbation $\Gamma$ of $\Gz$ because it establishes merely that the perturbation of an eigenvalue tends to zero as the size of the attached corner tends to zero.  One requires tighter control over the variation of the eigenvalues in order to guarantee that an infinite sequence of eigenvalues tending to zero is retained, with the same sign, when passing from $\Gz$ to $\Gamma$.  

A desirable result would be to prove that, for a symmetric curve $\Gamma$ with an outward-pointing corner, the positive part of $\KsGo$ is compact and has infinite rank.  This may not be unreasonable, seeing that $\KsGo$ has non-positive essential spectrum.  Such a result would guarantee an infinite sequence of positive eigenvalues of $\KsGo$ which would overlap with the essential spectrum of $\KsGe$. 

\smallskip
{\bfseries 2.} What happens when the essential spectrum of $\KsGe$ overlaps eigenvalues of $\KsGze$? 
We expect that such eigenvalues of $\KsGze$ would not be perturbed to eigenvalues of $\KsGe$ but rather would do the generic thing and become resonances, which are poles of the analytic continuation of the resolvent of $\KsGe$ onto another Riemann sheet.  This type of resonance is demonstrated numerically in~\cite[Fig.~6]{HelsingKangLim2016},  where one observes resonances around the spectral values $\pm0.08$; this example is discussed in more detail in point~5 below.

\smallskip
{\bfseries 3.} Can one construct embedded eigenvalues of the Neumann-Poincar\'e operator in the absence of reflectional symmetry?

\smallskip
{\bfseries 4.} The technique of perturbing a reflectionally symmetric $C^{2,\alpha}$ curve by attaching corners to create embedded eigenvalues is not extensible to a curve that admits a different group of symmetries, at least not in a straightforward manner.  Consider a curve $\Gamma$ with a finite cyclic rotational symmetry group $C_r$ of order~$r$.  The Neumann-Poincar\'e operator is decomposed on the $r$ orthogonal eigenspaces of the action of $C_r$ on $\Hmhalf(\Gamma)$, that is, the Hilbert-space decomposition
\begin{equation}\label{Cr1}
  H^{-1/2}(\Gamma) = H^{-1/2,0}(\Gamma) \oplus \cdots \oplus H^{-1/2,r-1}(\Gamma)
\end{equation}
into eigenspaces of $C_r$ induces a decomposition
\begin{equation}\label{Cr2}
   \KsG = {\mathcal K_{\Gamma,0}^*} \oplus \cdots \oplus {\mathcal K_{\Gamma,r-1}^*}\,.
\end{equation}
If $\Gamma$ has exactly $r$ small corners that are cyclically permuted under $C_r$, the essential spectrum of each of these component operators is a symmetric interval $[-b,b]$.  This is in contrast to the case of reflectional symmetry, as was seen earlier, where $\sigma_\ess(\KsGo)=[-b,0]$ and $\sigma_\ess(\KsGe)=[0,b]$ (for an outward-pointing corner); and in contrast to the rotationally invariant surface with a conical point in $\RR^3$ investigated by Helsing and Perfekt~\cite[Theorem~3.8,\,Fig.~5]{HelsingPerfekt2017}, in which different Fourier components of the Neumann-Poincar\'e operator have different essential spectrum.

\smallskip
{\bfseries 5.} 
What if a corner is attached to a smooth curve without smoothing out the points of attachment?  The additional corners at the attachment points will contribute to the essential spectrum of the Neumann-Poincar\'e operator of the perturbed domain.  A nice example of this is provided by a numerical computation of Helsing, Kang and Lim in~\cite[Fig.~6]{HelsingKangLim2016}.  There, the $C^{2,\alpha}$ curve is an ellipse $\Gz$, to which an outward corner is attached symmetrically with respect to the minor axis $L$ of symmetry of the ellipse to create a perturbed Lipschitz curve $\Gamma$, illustrated in Fig.~\ref{fig:ellipse}.  Two additional inward corners not lying on $L$ are created by this attachment, and these two corners are positioned symmetrically about $L$.  The computation in~\cite{HelsingKangLim2016} demonstrates exactly one positive embedded eigenvalue with odd eigenfunction and exactly one negative embedded eigenvalue with even eigenfunction.  In fact, this is expected based on the eigenvalues of $\KsGz$ and the essential spectrum of~$\KsG$.

Specifically, we will show that (i)~the essential spectrum of the even and odd components of $\KsG$, created by the three corners, are
\begin{equation}\label{ellipseessential}
  \renewcommand{\arraystretch}{1.1}
\left.
\begin{array}{lll}
  \sigma_\ess(\KsGo) &=& \textstyle[-\frac{1}{4},\frac{1}{8}-\eta], \\
  \sigma_\ess(\KsGe) &=& \textstyle[-\frac{1}{8}+\eta,\frac{1}{4}],    
\end{array}
\right.
\end{equation}
in which $\eta$ is a tiny number with $0<\eta<1/8$, (ii)~the largest four eigenvalues (see~\ref{alpha}) of $\KsGz$ are equal to $\pm\alpha_1=\pm 1/5$ and $\pm\alpha_2=\pm2/25$, and (iii)~the eigenfunction for eigenvalue $1/5$ is odd, and that for $-1/5$ is even.  Theorem~\ref{thm:main} and the supporting lemmas can be modified to handle this example, in which the perturbed part of the curve has more than one corner.  By making the corner attachment small enough so that $\KsGo$ has a (nonembedded) eigenvalue sufficiently near $1/5$ and $\KsGe$ has a (nonembedded) eigenvalue sufficiently near $-1/5$, these eigenvalues of~$\KsG$ are contained within the essential spectrum of~$\KsG$ in view of~(\ref{ellipseessential}).  And the corner attachment can be made small enough such that $\alpha_2=2/25<1/8-\eta$, so that the next eigenvalues in the sequence lie within the essential spectra of both $\sigma_\ess(\KsGo)$ and $\sigma_\ess(\KsGe)$ and thus are not expected to be perturbed to eigenvalues of $\KsG$.

Items (ii) and (iii) are results of the discussion on ellipses at the end of section~\ref{sec:embedded}, using $\varrho_0=\tanh^{-1}(3/7)$.
Item~(i) can be proved as follows.  Modify the proof of Proposition~\ref{prop:evenodd} by letting the cutoff function $\rho_1$ be localized about the one outward corner lying on~$L$ and letting $\rho_2=\rho_2^+ + \rho_2^-$ be a sum of two cutoff functions, one localized about each of the two inward corners not lying on~$L$.  As before, one has
\begin{equation}
  \KGe \;=\; \sum\limits_{0\leq i,j \leq n} M_{\rho_i} \KGe M_{\rho_j},
\end{equation}
with $\rho_0+\rho_1+\rho_2=1$, and the essential spectrum is
\begin{equation}
\sigma_\ess(\KGe) \;=\; \sigma_{\text{ea}} \left( M_{\rho_1} \KGe M_{\rho_1} \right) 
                   \cup \sigma_{\text{ea}} \left( M_{\rho_2} \KGe M_{\rho_2} \right).
\end{equation}
The half exterior angle of the outward corner is $\theta_1=3\pi/4$, and thus
$\sigma_{\text{ea}} \left( M_{\rho_1} \KGe M_{\rho_1} \right)$ is equal to the positive interval $[0,1/4]$ since this operator acts on functions that are even with respect to~$L$.  The operator $M_{\rho_2} \KGe M_{\rho_2}$ also acts on functions that are even with respect to~$L$, but since the inward corners do not lie on~$L$, the symmetry of a function about $L$ does not restrict the function near either of the inward corners.
Thus the contribution to the essential spectrum coming from the inward corners is the full interval $[-b,b]$, with $b=1/8-\eta>0$ since the half exterior angle is a little bigger than~$3\pi/8$; that is to say,
\begin{equation}
  \sigma_{\text{ea}} \left( M_{\rho_2} \KGe M_{\rho_2} \right)
  =
  \sigma_{\text{ea}} \left( M_{\rho_2^+} \KG M_{\rho_2^+} \right)
  =
  [-\textstyle\frac{1}{8}+\eta,\textstyle\frac{1}{8}-\eta].
\end{equation}
Likewise, $\sigma_{\text{ea}} \left( M_{\rho_2} \KGo M_{\rho_2} \right)=[-\textstyle\frac{1}{8}+\eta,\textstyle\frac{1}{8}-\eta]$.

To make rigorous the assumption above that the eigenvalues $\pm\alpha_1$ of $\KsGz$ are perturbed into eigenvalues of $\KsG$, notice that Lemma~\ref{lemma:resolvent} does not rely on the smoothness of the attachment of the replacement curve $D$ to $\Gz$, so the resolvent bound established by that Lemma holds for this example.  In view of the essential spectra~(\ref{ellipseessential}) of the even and odd components, one can establish the existence of the perturbed eigenvalues in a manner following the proof of Theorem~\ref{thm:main}.

\begin{figure}  
  \centerline{\scalebox{0.33}{\includegraphics{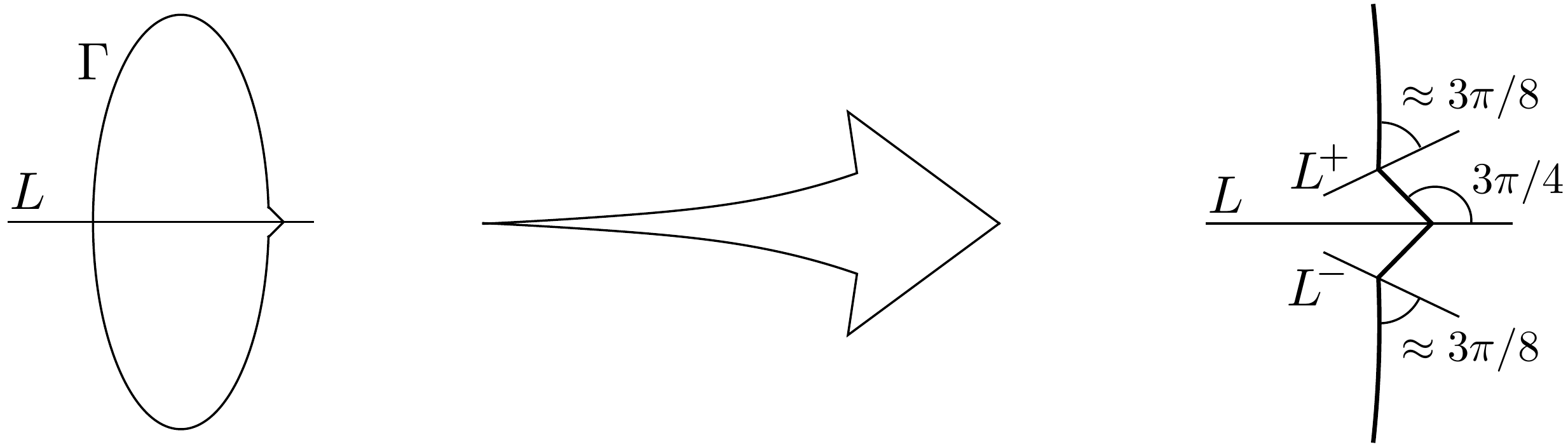}}}
\caption{\small This is the Lipschitz perturbation $\Gamma$ of an ellipse treated numerically in~\cite[Fig.~6]{HelsingKangLim2016}. An outward-pointing corner replaces a small section of the ellipse centered around its minor axis $L$.  The points at which the corner attaches to the ellipse introduce two inward-pointing corners.  The lines $L^{\!-}$ and $L^{\!+}$ bisect these two corners.}
\label{fig:ellipse}
\end{figure}

\bigskip
\noindent{\bfseries Acknowledgement.}
This material is based upon work supported by the National Science Foundation under Grant No. DMS-1814902.

\end{document}